\documentclass[twoside,11pt,reqno, final]{amsart}
\usepackage{graphicx}
\usepackage{upref,amsxtra,amssymb,amscd}
\usepackage{varioref}
\usepackage{verbatim}
\usepackage{epsfig}
\usepackage{color}
\usepackage{mathrsfs}

\usepackage{epic,eepic,eucal}
\usepackage{enumerate}

\def\Or{          \mathcal O}

\def\eLL{          \mathcal L}
\def\mN{          \mathcal N}
\def\sC{          \mathscr C}
\def\mS{          \mathcal S}
\def\mP{					\mathcal P}
\def\fp{					\mathfrak p}

\def\mF{					\mathcal F}
\def\mE{					\mathcal E}

\def\a{         \alpha}
\def\b{         \beta}
\def\g{         \gamma}

\newcommand{\NN}{{\mathbb N}}
\newcommand{\RR}{{\mathbb R}}
\newcommand{\HH}{{\mathbb H}}
\newcommand{\TT}{{\mathbb T}}

\newcommand{\ZZ}{{\mathbb Z}}

\newcommand{\QQ}{{\mathbb Q}}

\newtheorem{theo}{\sc Theorem}[section]
\newtheorem{prop}[theo]{\sc Proposition}

\newtheorem{lemm}[theo]{\sc Lemma}
\newtheorem{coro}[theo]{\sc Corollary}
\newtheorem{conj}[theo]{\sc Conjecture}

\theoremstyle{definition}

\theoremstyle{remark}

\newtheorem{rema}[theo]{\sc Remark}

\numberwithin{equation}{section}

\begin{document}
\title{Badly approximable vectors on rational quadratic varieties}

\author{Jimmy Tseng}
\address{Jimmy Tseng, Department of Mathematics, The Ohio State University, Columbus, OH 43210}
\email{tseng@math.ohio-state.edu}

\begin{abstract}
Approximation in this paper is of vectors on the unit $d$-cube by the projection of integer lattice points onto the same cube.  We define badly approximable vectors on a rational quadratic variety and show that sets of these vectors, which are (naturally) indexed by $m \in \QQ$, are winning and strong winning in the sense of Schmidt games.  From the winning property, it follows that these sets have full Hausdorff dimension and, moreover, so does their intersection.  In most cases, these sets are known to be null sets.
\end{abstract}

\maketitle

\tableofcontents
\listoffigures

\section{Introduction}

In~\cite{GS}, A.~Gorodnik and N.~Shah prove, for approximation on rational quadratic varieties, the analog of the Khinchin theorem, an archetypal and seminal result in the theory of Diophantine approximation that relates approximation to summation;\footnote{For the precise statement of the Khinchin theorem and its generalization, the Khinchin-Groshev theorem, see, for example, Theorems~1 and~2 of~\cite{Dod}.} the result of Gorodnik and Shah does likewise for rational quadratic varieties and provides the motivation for the main results of this paper on badly approximable vectors.

Let us introduce the notion of approximation on rational quadratic varieties and the Gorodnik-Shah theorem.  Let $X:=X_m := \{w \in \RR^{d+1} \mid Q(w) = m\}$ for some $m \in \QQ$ where $Q$ is a rational, nondegenerate, indefinite, quadratic form.  Let $\|\cdot\|_2$ be the Euclidean norm and $\|\cdot\|$ be the sup norm on $\RR^{d+1}$.  Then define \[\partial X := \{x \in \RR^{d+1} \mid Q(x) = 0\} \cap C^d,\] where $C^d$ is the unit $d$-cube in $\RR^{d+1}$ (i.e. $C^d = \{ v \in \RR^{d+1} \mid \|v\|=1\}$).  Let $\fp:  \RR^{d+1} \backslash \{\mathbf{0}\}  \rightarrow C^d; x \mapsto \frac {x}{\|x\|}$ be radial projection, and let $\psi:(0, \infty) \rightarrow (0, \infty)$ be a measurable, quasi-conformal function.  Then, for any $Z \subset \ZZ^{d+1}$, we say, following~\cite{GS}, that a vector $v \in C^d$ is \textit{$(Z, \psi)$-approximable} if the inequality \[\|\fp(x) - v\| < \psi(\|x\|)\] has infinitely many solutions $x \in Z \backslash \{\mathbf{0}\}$.  Note that all vectors are first projected by $\fp$ onto $C^d$ before any approximation takes place and that, since $\partial X$, our main space of study, is the set of points in which a rational quadratic variety meets the unit cube, we speak of approximating points of $\partial X$ as approximation on a rational quadratic variety. 
Next define \[X(\ZZ) :=X_m(\ZZ) := X \cap \ZZ^{d+1}.\]  Typically, we shall discuss the $(X(\ZZ), \psi)$-approximability of vectors in $\partial X$.

Let $d\geq3$ and $G = O(Q)$ be the orthogonal group given by the quadratic form $Q$.   
The group $G$ acts on $C^d$ as follows:  $g \cdot \fp(w) = \fp(g w)$ where $w \in \RR^{d+1}$.  Under this action, $\partial X$ is a homogeneous space of $G$ and admits a unique $G$-semi-invariant probability measure $\mu_\infty$~\cite{GS}.  Then the part of the Gorodnik-Shah theorem, Theorem~1.2(i) of~\cite{GS}, that forms the background for us is the following:\footnote{In~\cite{GS}, the unit cube is replaced by the unit sphere, which has no effect on the result.  Thanks to N.~Shah for pointing this out.}

\begin{theo}\label{theoGS}
Let $d\geq3$ and $m \neq 0$.  Let $\psi:(0, \infty) \rightarrow (0, \infty)$ be a measurable quasi-conformal function.  If \[\int_1^\infty t^{d-2} \psi(t)^{d-1} dt = \infty,\] then $\mu_\infty$-a.e. $v \in \partial X$ is $(X_m(\ZZ), \psi)$-approximable.
\end{theo}

\subsection{Badly approximable vectors}  Let $Z \subset \ZZ^{d+1}$.  In this paper, we study the following set \begin{align*}BA_{\partial X}^\psi(Z) :=  \{v \in \partial X  \mid  &\textrm{ there exists } c(v) > 0 \textrm{ such that, for all } x \in Z \backslash \{0\}, \\ &\|\fp(x) - v \| \geq c \psi(\|x\|)\},\end{align*} which we denote the \textit{set of badly $(Z, \psi)$-approximable vectors of $\partial X$} (and, informally, as the {\it set of badly approximable vectors}).  When $m \neq 0$ (in high enough dimensions), Theorem~\ref{theoGS} implies that $\mu_\infty(BA_{\partial X}^\psi(X_m(\ZZ))) =0$ for all those $\psi$ for which the condition of the theorem holds (the function $\psi(t) = t^{-1}$--our primary concern--is an example);\footnote{When $m =0$, there is no (known) analog of Theorem~\ref{theoGS}; however, this lack is immaterial for our result, Theorem~\ref{thmBALightCone}, because it does not matter from the point of view of Schmidt games (see Section~\ref{secSchmidtgames}) whether the strong winning set is $\mu_\infty$-null or not--and our result would not be trivial even if the set has full measure.} Example 5.1 of~\cite{GS}, however, shows that $BA_{\partial X}^\psi(X_m(\ZZ))$ can be nonempty for $\psi(t) = t^{-1}$. 

While badly approximable vectors on rational quadratic varieties have not been studied before (as far as the author knows), the--roughly speaking--dual object, very well approximable vectors, have been studied by C. Dru\c{t}u in~\cite{Dr} (see the Introduction of that paper for definitions).  In particular, the Hausdorff dimension of those sets are computed for certain $\psi(t)$ (Theorem~1.1 of~\cite{Dr}).

\subsection{Schmidt games}\label{secSchmidtgames}
W.~Schmidt introduced the games which now bear his name in~\cite{Sch2}.  This game and its variant are our main tools.  Let $S$ be a subset of a complete metric space $M$.  For any point $x \in M$ and any $r \in \RR_{>0}$, we denote the closed ball in $M$ around $x$ of radius $r$ by $B(x, r)$.  Even though it is possible for there to exist another $x' \in M$ and $r' \in \RR_{>0}$ for which $B(x,r) = B(x',r')$ as sets in $M$, there is no ambiguity for us, as we always assume that we have chosen (either explicitly or implicitly) a center and a radius for each closed ball.  Given a closed ball $W$, let \begin{align*}\rho(W) & \textrm{ denote its radius and} \\ c(W) & \textrm{ denote its center.}\end{align*}  

Schmidt games require two parameters:  $0 < \a <1$ and $0 < \b<1$.  Once values for the two parameters are chosen, we refer to the game as the $(\a,\b)$-game, which we now describe.  Two players, Player $B$ and Player $A$, alternate choosing nested closed balls $B_1 \supset A_1 \supset B_2 \supset A_2 \cdots$ on $M$ such that \begin{eqnarray}\label{eqnSchmidtGameRadiusRules}\rho(A_n) = \a \rho(B_n) \textrm{ and }\rho(B_n) = \b \rho(A_{n-1}).\end{eqnarray}  The second player, Player $A$, \textit{wins} if the intersection of these balls lies in $S$.\footnote{Completeness of a metric space is equivalent to the nested closed sets property:  thus this intersection is exactly one point.}  A set $S$ is called \textit{$(\a, \b)$-winning} if Player $A$ can always win for the given $\a$ and $\b$.  A set $S$ is called \textit{$\a$-winning} if Player $A$ can always win for the given $\a$ and any $\b$.  A set $S$ is called \textit{winning} if it is $\a$-winning for some $\a$.  
Schmidt games have two important properties for us~\cite{Sch2}: \medskip


\noindent$\bullet$ Countable intersections of $\a$-winning sets are again $\a$-winning.




\noindent$\bullet$ The sets in $\RR^m$ which are $\a$-winning have full Hausdorff dimension. \medskip


Recently, C.~McMullen defined in~\cite{Mc} a variant of the game:  strong-winning Schmidt games.  To define this variant, we modify Schmidt games as follows:  replace the requirement on radii of balls as stated in (\ref{eqnSchmidtGameRadiusRules}) with \[\rho(A_n) \geq \a \rho(B_n) \textrm{ and }\rho(B_n) \geq \b \rho(A_{n-1}).\]  Using this modification, the notions of $(\a,\b)$-strong winning, $\a$-strong winning, and strong winning for the subset $S$ are defined in the analogous way.\footnote{The intersection of the players' balls may contain more than one point, but Player $A$, by judicious choice of radii, can force the intersection to contain exactly one point.}  On compact metric spaces ($\partial X$ for example), strong winning is preserved by quasisymmetric homeomorphisms~\cite{Mc} (but see Remark~\ref{remaStrongWinningQuasiSym} for more about quasisymmetric homeomorphisms on metric spaces other than Euclidean spaces), while winning is merely preserved by bilipschitz homeomorphisms (see Lemma~\ref{lemmLocalBilipWinningSetsweak} and its remark and Theorem~1.1 of~\cite{Mc}).   Since bilipschitz homeomorphisms are quite a restrictive subclass of quasisymmetric homeomorphisms, strong-winning has considerable benefits over winning.  Moreover, strong-winning Schmidt games have the same two properties~\cite{Mc}:\medskip

\noindent$\bullet$ Countable intersections of $\a$-strong winning sets are again $\a$-strong winning.

\noindent$\bullet$ The sets in $\RR^m$ which are $\a$-strong winning have full Hausdorff dimension.

\subsection{Conjecture }


With Theorem~\ref{theoGS}, the aforementioned example in~\cite{GS}, and the analogy with the usual notion of badly approximable vectors in Euclidean space as supporting evidence, we conjecture that, for $\psi(t) = t^{-1}$, the set $BA_{\partial X}^\psi(X_m(\ZZ))$ (which is a null set for $m \neq 0$, as mentioned) has plenty of points:

\begin{conj} 
Let $\psi(t) = t^{-1}$ and $m \in \QQ$.  Then $BA_{\partial X}^\psi(X_m(\ZZ))$ is a (strong) winning subset of $\partial X$.
\end{conj}

\noindent Our main results (Theorems~\ref{thmBALevelSurface} and~\ref{thmBALightCone}) prove the conjecture.

\subsection{Statement of results}  For the usual Diophantine approximation (in Euclidean space), the analog of the conjecture is a classical and important result.\footnote{The winning assertion is classical and due to Schmidt~\cite{Sch3}.  For strong winning, see Section 5.1~of~\cite{ET}.}  We show in Section~\ref{secProofofConjecture} that the result also holds for rational quadratic varieties.



Let $m \in \QQ$ be as above.  The proof of the conjecture is different for $m=0$, the {\it light-cone case}, and $m \neq 0$, the {\it level-surface case}.  For the light-cone case, an alternate proof using group actions follows (with a little work) from~\cite{KW}.\footnote{Thanks to D.~Kleinbock for pointing this out.}  Our proof is different:  we do not use group actions, only geometry.  Our proof of the level-surface case is new.  We show the following:

\begin{theo}\label{thmBALevelSurface}
Let $\psi = t^{-1}$, $m \neq 0$, and $d \geq 2$.  Then  $BA_{\partial X}^\psi(X_m(\ZZ))$ is $\a$-strong winning and $\a$-winning.
\end{theo}

Using Lemmas~\ref{lemmVarietyHasBilpMaps} and~\ref{lemmWinningFullHDManifolds}, it follows that 
\begin{coro}\label{corBALevelSurface}
Let $\psi = t^{-1}$, $m \neq 0$, and $d \geq 2$.\footnote{For $d=1$, Theorem~\ref{thmDrCor} gives an answer for certain $Q$ and $m$.}  Then  $BA_{\partial X}^\psi(X_m(\ZZ))$ has full Hausdorff dimension (i.e. $= d-1$).
\end{coro}

As we shall see, an adaptation of the proof of Theorem~\ref{thmBALevelSurface} yields

\begin{theo}\label{thmBALightCone}
Let $\psi = t^{-1}$, $m=0$, and $d \geq 2$.\footnote{For $d=1$, Theorem~\ref{thmDrCor} gives a complete description.  Note that, since all winning subsets are dense, the only winning subset of a discrete metric space is the whole space.}   Then  $BA_{\partial X}^\psi(X_0(\ZZ))$ is $\a$-strong winning and $\a$-winning.
\end{theo}

Using the same lemmas, it follows that 
\begin{coro}\label{corBALightCone}
Let $\psi = t^{-1}$, $m=0$, and $d \geq 2$.  Then  $BA_{\partial X}^\psi(X_0(\ZZ))$ has full Hausdorff dimension (i.e. $= d-1$).
\end{coro}

Finally, a model corollary, which follows immediately from these same lemmas and the properties of Schmidt games, is

\begin{coro}\label{coroModelGamesBARatVariety}
Let $\psi = t^{-1}$ and $d \geq 2$.  Then  $\cap_{m \in \QQ} BA_{\partial X}^\psi(X_m(\ZZ))$ is $\a$-strong winning, $\a$-winning, and has full Hausdorff dimension (i.e. $= d-1$).
\end{coro}

\begin{rema}\label{remaStrongWinningQuasiSym}
Strong winning (and absolute winning, to be mentioned in the Conclusion) are preserved under a general class of homeomorphisms, which on Euclidean space are called quasisymmetric.  The conditions on these mappings for any complete metric space are enumerated in Section 2 of~\cite{Mc} (one distinguished subclass is composed of bilipschitz homeomorphisms).  For more details, see Theorems 1.2, 2.1, and 2.2 and the final remark of Section 2 from that paper.  Using these results from~\cite{Mc}, it follows immediately that, for any countable family of $k_i$-quasisymmetric homeomorphisms with uniformly bounded constants (namely that there exists a constant $k$ such that $k_i \leq k$ for all $i$), the intersection of their images of the set in Corollary~\ref{coroModelGamesBARatVariety} is still strong winning.  To make the analogous (but weaker) statement with $k_i$-bilipschitz homeomorphisms and winning, one can use Lemma~\ref{lemmLocalBilipWinningSetsweak} and its footnote.
\end{rema}

\noindent Note that $\a$ is a constant depending only on $d$ for diagonal (rational, nondegenerate, indefinite) quadratic forms; for its value, see the beginning of Section~\ref{secProofLevelSurfaceBA}.  For arbitrary (rational, nondegenerate, indefinite) quadratic forms, $\a$ depends also on the form; see Remark~\ref{rmkGeneralQuadForm} for its value.  Also note that $\a$ can be replaced by $1/2$ in all cases; see the Conclusion.

\subsubsection{Auxiliary results}  We have three auxiliary results that complement and provide context for our aforementioned results; these are proved in Section~\ref{secAuxObsonBA}.  The smallest possible dimension $d$ that makes sense for approximation is $d=1$.  For this dimension, $\partial X$ is a finite set, and we greatly strengthen our main results (for most cases).  To state this strengthening, let us assume, without loss of generality, that $Q = q - y^2$ where $q(x) = \a x^2$ (note that, since Q is indefinite, $\a > 0$, and, moreover, by renaming the variables and multiplying by $-1$ if necessary, we may assume, without loss of generality, that $\a\geq1$).\footnote{Here (in Theorem~\ref{thmDrCor}), we are further assuming that $Q$ is a diagonal (rational, nondegenerate, indefinite) quadratic form; for an arbitrary (rational, nondegenerate, indefinite) quadratic form, one can follow the proof in Remark~\ref{rmkGeneralQuadForm} and perform the analogous changes to the proof of Theorem~\ref{thmDrCor} in Section~\ref{secAuxObsonBA}.}  Then $\partial X$ is the four-element set $(\pm 1/\sqrt{\a}, \pm 1)\}$. 
We give a simple proof of the following:

\begin{theo}\label{thmDrCor}
Let $\psi(t) = t^{-1}$.  Let $d=1$ and $q(x) = \a x^2$. \begin{enumerate}
\item If $\sqrt{\a}$ is rational, then, for all $m \neq 0$, $BA_{\partial X}^\psi(X_m(\ZZ))= \partial X$ and, for $m=0$,  $BA_{\partial X}^\psi(X_0(\ZZ))= \emptyset$.
\item If $\sqrt{\a}$ is irrational, then, for $m$ small enough in absolute value (depending on $\a$), $BA_{\partial X}^\psi(X_m(\ZZ))= \partial X$.
\end{enumerate}

\end{theo}

The analogs of Theorems~\ref{thmBALevelSurface},~\ref{thmBALightCone}, and~\ref{thmDrCor} for $\psi(t)=t^{-s}$ where $s >1$ are immediate from those theorems (in the case of Theorem~\ref{thmDrCor} in which $m=0$ and $\sqrt{\a}$ is rational, this follows because the approximation is exact).\footnote{Since a set of full measure need not be winning, the analogs of Theorems~\ref{thmBALevelSurface} and \ref{thmBALightCone} for $s>1$ are not trivial.}  The complicated proof of Theorem~\ref{thmBALevelSurface}, however, is not necessary for $s >2$ and $m \neq 0$.  Using a simple argument, we show 

\begin{theo}\label{thmApproximationNotPossSbigger2}
Let $m \neq 0$ and $\psi(t) = t^{-s}$.  For $s \in (2, \infty)$, $BA_{\partial X}^\psi(X_m(\ZZ)) = \partial X$.
\end{theo}

\noindent Moreover, approximation in the case of Theorem~\ref{thmApproximationNotPossSbigger2} is not meaningful, as described in Remark~\ref{remaKleinbocksRemark}.

Finally, note that there are two natural sets of integer lattices points to approximate with:  $X(\ZZ)$ and $\ZZ^{d+1}$.   A priori, it may be possible that $BA_{\partial X}^\psi(\ZZ^{d+1})$ (for $\psi(t) =t^{-1}$) is already quite large;  we show, however, that this is not the case:

\begin{theo}\label{theoFullLatticeTooBig}
Let $\psi(t) = t^{-s}$ where $ s \in [0,1]$.  Then $BA_{\partial X}^\psi(\ZZ^{d+1})$ is empty.
\end{theo}

\noindent This last theorem suggests that $X(\ZZ)$ is the natural set of integer lattice points to (badly) approximate with (at least for $\psi(t) =t^{-1}$) and that the geometry of the quadratic variety significantly affects the set of badly approximable vectors.

\subsection*{Acknowledgements}  I would like to thank Nimish Shah for pointing me to~\cite{GS}, for his helpful comments, and for his encouragement.  I would also like to thank Dmitry Kleinbock for stimulating discussions during his brief visit to Ohio State in June 2010, for continuing helpful discussions, and for encouragement.

\section{Proof of Conjecture}\label{secProofofConjecture}

The proof of the level-surface case is in Section~\ref{secProofLevelSurfaceBA}; the light-cone case, in Section~\ref{secProofLightConeBA}.  We begin with common notation and lemmas.  


\subsection{Notation} There is a natural splitting of $Q = q_1 - q_2$ where $q_1 := a_1  Y_1^2 + \cdots + a_k  Y_k^2$ and $q_2 := a_{k+1}  Y_{k+1}^2 + \cdots + a_{d+1}  Y_{d+1}^2$ are both positive-definite rational quadratic forms (i.e. $0 < a_i \in \QQ$).\footnote{The choice of a diagonal quadratic form here is without loss of generality because an arbitrary (rational, nondegenerate, indefinite) quadratic form $\tilde{Q}$ is equivalent to some diagonal (rational, nondegenerate, indefinite) quadratic form $Q$ (see Corollary~7.30 of~\cite{EKM}) and because the proof for $\tilde{Q}$ is virtually the same as the proof for $Q$ (see Remark~\ref{rmkGeneralQuadForm}).}   Let $a_i = \tilde{a_i}/s$ where $s$ is the least common multiple of the denominators of all of the $a_i$s (written in lowest terms).  Let $m \in \QQ$ be as in the Introduction.

The natural splitting of $Q$ corresponds to the direct sum $\RR^{d+1} = \RR^{k} \oplus \RR^{d-k+1}$ such that a vector $\langle w, u \rangle \in \RR^{d+1}$ is uniquely written as a vector $w$ (the $q_1$-component) in the (ordered) coordinates $ Y_1, \cdots,  Y_k$ and a vector $u$ (the $q_2$-component) in the (ordered) coordinates $ Y_{k+1}, \cdots,  Y_{d+1}$.  The splitting also yields two norms:  $\| \cdot\|_{q_1} := \sqrt{q_1(\cdot)}$ and $\| \cdot\|_{q_2} := \sqrt{q_2(\cdot)}$.  These norms satisfy a key relation for any element $\langle w, u \rangle \in X_m$:  \begin{eqnarray}\label{eqnnormrelation} \| w\|_{q_1}^2 -  \| u\|_{q_2}^2=m.\end{eqnarray}  Also, note that the balls of a fixed radius given by either norm are bounded convex sets and hence contain a finite number of integer lattice points.

Approximation in this context is by integer lattice points on $X_m$.  We partition this set of integer lattice points $X_m(\ZZ)$ in two ways:  the first is to collect the elements with the same $q_1$-components into the same coset \[X^{\|w\|_{q_1}}:= \{\langle w, u \rangle \in \ZZ^{d+1} \mid Q(\langle w, u \rangle) =m\}\] and the other is to collect the same $q_2$-components into the same coset \[X_{\|u\|_{q_2}}:= \{\langle w, u \rangle \in \ZZ^{d+1} \mid Q(\langle w, u \rangle) =m\}.\]  By (\ref{eqnnormrelation}), either type of coset has finite cardinality.  For $m=0$, these two ways of partitioning are identical; for $m \neq 0$, the distinction does not matter as either norm grows large.  

For our proof, we are only concerned with unions of cosets (over ranges of the $q_1$ or $q_2$-components, respectively); it is, as it will become evident, convenient to introduce the following notation:  a vector $\langle w, u \rangle$ is in the following union of cosets \[\bigcup_{C_1 \leq \|U\|_{q_2} \leq C_2} X_{\|U\|_{q_2}}\] if $\langle w, u \rangle \in X_{\|u\|_{q_2}}$ and $C_1 \leq \|u\|_{q_2} \leq C_2$ where $C_1$ and $C_2$ are constants, and likewise for the other type of partitioning.

Since we project vectors onto the unit cube, we cannot distinguish between multiplies; thus, given two vectors $v, v' \in \RR^\ell$, define $v \sim v'$ if there exists a nonzero real number $\g$ such that $v = \gamma v'$.  Two elements of $X(\ZZ)$ equivalent under $\sim$ are the same for us.  For the proof, however, we need three other (finer) equivalence relations (all of which are related to the natural splitting of $Q$).  
Define the equivalence relation $\approx$ on $\RR^{d+1}$ as follows:  $\langle w, u \rangle \approx \langle w', u' \rangle$ if there exists two nonzero real numbers $\g, \tilde{\g}$ such that $w = \g w'$ and $u = \tilde{\g} u'$.  Define the equivalence relation $\sim_1$ on $X(\ZZ)$ as follows\footnote{The subset of $X(\ZZ)$ where the $q_2$-component is the zero vector is, at most, a finite set, and we may put all of these elements into the same equivalence class; however, this class is immaterial for the proof.}:  $\langle w, u \rangle \sim_1 \langle w', u' \rangle$ if \[\frac{w}{\|u\|_{q_2}} = \frac{w'}{\|u'\|_{q_2}}.\]  And, likewise define, $\sim_2$ on $X(\ZZ)$ as follows:  $\langle w, u \rangle \sim_2 \langle w', u' \rangle$ if \[\frac{u}{\|w\|_{q_1}} = \frac{u'}{\|w'\|_{q_1}}.\]

Finally, besides the norms $\| \cdot\|_{q_1}$ and $\| \cdot\|_{q_2}$, we also use the sup norm $\| \cdot \|$, the usual Euclidean $2$-norm $\| \cdot \|_2$, and the norm on any vector $\langle w, u \rangle \in \RR^{d+1}$ given by $\|w\|_{q_1}+ \|u\|_{q_2}$.  Since all norms on $\RR^{d+1}$ are equivalent, we have that there exists a constant $c_s \geq 1$ (depending only on the norms) such that\footnote{The subscript $s$ is shorthand for the sum norm (given by the natural splitting) and the sup norm; it is not related to the least common multiple $s$.} \[\frac 1 {c_s} (\|w\|_{q_1}+ \|u\|_{q_2}) \leq \|\langle w, u \rangle\| \leq c_s (\|w\|_{q_1}+ \|u\|_{q_2}).\]  Likewise, there exists a constant $c_{2q_1} \geq 1$ such that $c_{2q_1}^{-1} \|\cdot\|_2 \leq \|\cdot\|_{q_1} \leq c_{2q_1} \|\cdot\|_2$; and, analogously, a constant $c_{2q_2} \geq 1$.


Define the positive constant $\kappa_0 := \frac{8 (\sqrt{10s}) c_s(1 + 1/\sqrt{3}) \max(c_{2q_1},c_{2q_2})}{\min(c_{q_1}, c_{q_2})}$ where the constants $c_{q_1}$ and $c_{q_1}$ are cases of the constant from Lemma~\ref{lemmEllipsoidMissesLine} for (and depending only on) $q_1$ and $q_2$, respectively.  (Note that $\kappa_0$ depends only on $Q$.)

\subsection{Integer lattice points repel}  To use Schmidt games, we must show that undesirable elements of certain subsets of the metric space that we are playing on repel each other--this is the key step to the use of games.  For our proof, the undesirable elements are the projections of elements of $X_m(\ZZ)$ onto the unit cube.  The precise statements that we need are

\begin{prop}\label{propIntegerSepartionLW} Let $m \neq 0$ and $K$ be a real number $ \geq 3\sqrt{|m|}$.  For any \[\langle w, u \rangle \not\approx \langle w', u' \rangle \in \bigcup_{3\sqrt{|m|} \leq \|U\|_{q_2} \leq K} X_{\|U\|_{q_2}},\] we have \[\bigg\|\frac{\langle w, u \rangle}{\|\langle w, u \rangle\|}-\frac{\langle w', u' \rangle}{\|\langle w', u' \rangle\|}\bigg\|_2 \geq \frac {8}{K \kappa_0}.\]
\end{prop}

\begin{prop}\label{propIntegerSepartionLC} Let $m =0$ and $K$ be a real number $ \geq 1$.  For any \[\langle w, u \rangle \not\sim \langle w', u' \rangle \in \bigcup_{1 \leq \|U\|_{q_2} \leq K} X_{\|U\|_{q_2}},\] we have \[\bigg\|\frac{\langle w, u \rangle}{\|\langle w, u \rangle\|}-\frac{\langle w', u' \rangle}{\|\langle w', u' \rangle\|}\bigg\|_2 \geq \frac {8}{K \kappa_0}.\]
\end{prop}

The key idea needed to prove these propositions is to use the natural splitting given by $Q$:

\begin{lemm}\label{lemIntegerSepartionLW}Let $m \neq 0$ and $K$ be a real number $ \geq 2\sqrt{|m|}$.  For any \[\langle w, u \rangle, \langle w', u' \rangle \in \bigcup_{2\sqrt{|m|} \leq \|U\|_{q_2} \leq K} X_{\|U\|_{q_2}}\] such that $w \not\sim w'$, we have \[\bigg\|\frac{w}{\|\langle w, u \rangle\|}-\frac{w'}{\|\langle w', u' \rangle\|}\bigg\|_2 \geq \frac {16}{K \kappa_0}.\]
\end{lemm}

\begin{lemm}\label{lemIntegerSepartionLC}Let $m =0$ and $K$ be a real number $ \geq 1$.  For any \[\langle w, u \rangle \not\sim_1 \langle w', u' \rangle \in \bigcup_{1\leq \|U\|_{q_2} \leq K} X_{\|U\|_{q_2}},\] we have \[\bigg\|\frac{w}{\|\langle w, u \rangle\|}-\frac{w'}{\|\langle w', u' \rangle\|}\bigg\|_2 \geq \frac {16}{K \kappa_0}.\]
\end{lemm}

\begin{proof}[Proof of Lemma~\ref{lemIntegerSepartionLW}]  Let $\ell := d+1 -k$.  Since $2\sqrt{|m|} \leq \|u\|_{q_2}$, there exists a vector $\tilde{u} \in \RR^\ell$ such that \begin{eqnarray} \label{eqnuApproxBnds} \frac 3 4 \|u\|_{q_2}^2 \leq \|\tilde{u}\|_{q_2}^2\leq \frac 5 4  \|u\|_{q_2}^2 \end{eqnarray} and \[\|\tilde{u}\|_{q_2}^2 =  \|u\|_{q_2}^2  + m.\]  In the analogous way, there exists a vector $\tilde{u}' \in \RR^\ell$.  By (\ref{eqnnormrelation}), we have \begin{eqnarray} \label{eqnEqualNorms}\|w\|_{q_1} = \|\tilde{u}\|_{q_2} \quad \textrm{ and } \quad \|w'\|_{q_1} = \|\tilde{u}'\|_{q_2}\end{eqnarray} and also have \[\frac{q_1(w)}{\|\tilde{u}\|_{q_2}^2} = 1 \quad \textrm{ and } \quad \frac{q_1(w')}{\|\tilde{u}'\|_{q_2}^2} = 1.\]  Since every ray emanating from the origin determines a vector in $\RR^\ell$, every ray must intersect the boundary of the closed unit $\|\cdot\|_{q_1}$-ball in $\RR^\ell$ (this ball is clearly bounded since it can be put into a big enough sup norm ball).  By the scalar multiplicativity property of norms, the intersection point is unique.  Since $w \not\sim w'$, it follows that the two unit vectors $\frac{w}{\|\tilde{u}\|_{q_2}}$ and $\frac{w'}{\|\tilde{u}'\|_{q_2}}$ are distinct, and hence we have that \begin{align*}0 & \neq \bigg\|\frac{w}{\|\tilde{u}\|_{q_2}}-\frac{w'}{\|\tilde{u}'\|_{q_2}}\bigg\|_{q_1}^2  = q_1\bigg(\frac{w}{\|\tilde{u}\|_{q_2}}-\frac{w'}{\|\tilde{u}'\|_{q_2}}\bigg) \\ & = a_1 \frac{w_1^2}{\|\tilde{u}\|_{q_2}^2} - \frac{2 a_1 w_1 w'_1}{\|\tilde{u}\|_{q_2} \|\tilde{u}'\|_{q_2}} + a_1 \frac{{w'}_1^2}{\|\tilde{u}'\|_{q_2}^2} + \cdots + a_k \frac{w_k^2}{\|\tilde{u}\|_{q_2}^2} - \frac{2 a_k w_k w'_k}{\|\tilde{u}\|_{q_2} \|\tilde{u}'\|_{q_2}} + a_k \frac{{w'}_k^2}{\|\tilde{u}'\|_{q_2}^2}\\ & = \frac{q_1(w)}{\|\tilde{u}\|_{q_2}^2} + \frac{q_1(w')}{\|\tilde{u}'\|_{q_2}^2}  - \frac 2 {s \|\tilde{u}\|_{q_2} \|\tilde{u}'\|_{q_2}} (\tilde{a}_1 w_1 w'_1 + \cdots + \tilde{a}_k w_k w'_k)\end{align*} where the $w_i$s are the components of $w$ and the $w'_i$s are the components of $w'$.  Since the norm is not zero and since $w$ and $w'$ are integer vectors (i.e. integer lattice points), we have \[\bigg\|\frac{w}{\|\tilde{u}\|_{q_2}}-\frac{w'}{\|\tilde{u}'\|_{q_2}}\bigg\|_{q_1} \geq \sqrt{\frac 2 {s \|\tilde{u}\|_{q_2} \|\tilde{u}'\|_{q_2}}} \geq \sqrt{\frac 8 {5 s K^2 }}\] where the last inequality follows from (\ref{eqnuApproxBnds}).

Again by (\ref{eqnuApproxBnds}) and by (\ref{eqnEqualNorms}), we have real constants $c$ and $c'$ such that \begin{eqnarray}\label{eqnequivnorms2}\|\tilde{u}\|_{q_2} = c \|\langle w, u \rangle\| \quad \textrm{ and } \quad \|\tilde{u}'\|_{q_2} = c' \|\langle w', u' \rangle\|\end{eqnarray} where $\frac 1{c_s (1 + 2/\sqrt{3})} \leq c, c' \leq \frac {c_s} {1 + 2/\sqrt{5}} $.  

Let $c_{q_1}>0$ be the constant, which depends only on $q_1$, from Lemma~\ref{lemmEllipsoidMissesLine}; then, that lemma implies that \[\bigg\|\frac w{\|\langle w, u \rangle\|}-\frac {w'}{\|\langle w', u' \rangle\|}\bigg\|_{q_1}= \ \bigg\|\frac w{\|\langle w, u \rangle\|}-\g c \frac{w'}{\|\tilde{u}'\|_{q_2}}\bigg\|_{q_1}  \geq c c_{q_1} \sqrt{\frac 8 {5 s K^2 }} \geq \frac{16}{\kappa_0 K}\] where $\g = c'/c$.
\end{proof}

\begin{proof}[Proof of Lemma~\ref{lemIntegerSepartionLC}]  This proof is just a simplification of the proof of Lemma~\ref{lemIntegerSepartionLW}; note that (\ref{eqnuApproxBnds}) is superfluous and $\tilde{u} = u$, $\tilde{u}'=u'$.  The rest of the proof is identical.  \end{proof}

\begin{proof}[Proof of Proposition~\ref{propIntegerSepartionLW}]  By (\ref{eqnEqualNorms}), it follows that \[\|w\|_{q_1} \leq \sqrt{\|u\|_{q_2}^2 + |m|} =\bigg\|(\|u\|_{q_2},  \sqrt{|m|})\bigg\|_2 \leq \|u\|_{q_2} + \sqrt{|m|}\] using the triangle inequality; whence, $\|w\|_{q_1} \leq 2 K$.  Again by  (\ref{eqnEqualNorms}), we have that $\|\|w\|^2_{q_1} \geq m + 9|m| \geq 8|m|$.  Thus, we have $2 \sqrt{|m|} \leq \|w\|_{q_1} \leq 2 K$.  The same bounds hold for $w'$.

Since $\langle w, u \rangle \not\approx \langle w', u' \rangle$, either $w \not\sim w'$ or $u \not\sim u'$ (or both can hold).  If $w \not\sim w'$, then Lemma~\ref{lemIntegerSepartionLW} implies the desired result in this case. 

If $u \not\sim u'$, then the desired result is also a consequence of the lemma.  First, note that the vectors of $\RR^{d+1}$ that satisfy $Q = m$ are same as those that satisfy $-Q = -m$.  Therefore the coset $X^{\|w\|_{q_1}}$ remains the same subset of $\ZZ^{d+1}$; in the same way, the coset $X_{\|u\|_{q_2}}$ remains the same.  The only difference between $Q=m$ and $-Q=-m$ is that $q_2$ is the positive-definite part of $-Q$ and $q_1$ is the negative-definite part; therefore, the roles of $q_1$ and $q_2$ are reversed in Lemma~\ref{lemIntegerSepartionLW} and $m$ is replaced by $-m$.  The latter does not affect the lemma since the conclusion depends only on the absolute value of $m$.  The bounds, however, for $w$ and $w'$, as noted above, are different:  $K$ is replaced with $2K$.  Therefore, the conclusion of the lemma in this case is as follows:  \[\bigg\|\frac{u}{\|\langle w, u \rangle\|}-\frac{u'}{\|\langle w', u' \rangle\|}\bigg\|_2 \geq \frac {8}{K \kappa_0},\] which implies the desired result.
\end{proof}

\begin{proof}[Proof of Proposition~\ref{propIntegerSepartionLC}]  This proof is just a simplification of the proof of Proposition~\ref{propIntegerSepartionLW}.  Since $\|w\|_{q_1} = \|{u}\|_{q_2} \textrm{ and } \|w'\|_{q_1} = \|{u}'\|_{q_2}$, we have the same bounds on the $q_1$-components as on the $q_2$.  Then the applications (for $\langle w, u \rangle \not\sim_1 \langle w', u' \rangle$ and $\langle w, u \rangle \not\sim_2 \langle w', u'\rangle$, respectively) of Lemma~\ref{lemIntegerSepartionLC} in the stead of Lemma~\ref{lemIntegerSepartionLW} is even easier. 

The remaining case to consider is when both $\langle w, u \rangle \sim_1 \langle w', u' \rangle$ and $\langle w, u \rangle \sim_2 \langle w', u'\rangle$ hold; this implies that \[w = \frac{ \|{u}\|_{q_2}}{\|{u}'\|_{q_2}} w' \quad \textrm{ and } \quad u = \frac{\|w\|_{q_1}}{\|w'\|_{q_1} }u'.\]  And thus $\langle w, u \rangle \sim \langle w', u'\rangle.$

\end{proof}

\subsection{Proof of Theorem~\ref{thmBALevelSurface}}\label{secProofLevelSurfaceBA}

In this section, we prove the level-surface case; the light-cone case, which is a simplification of this proof, we prove in Section~\ref{secProofLightConeBA}.


We begin the proof by playing a strong $(\a,\b)$-game on $\partial X$ for $\a :=1/(8c_\pi^2)$ and some $0 < \b <1$, where $c_\pi \geq 1$ is the bilipschitz constant for the radial projection (the map $\pi$ defined in Section~\ref{subsubsecHaCoC}) of the $1/2$ thickening of a face of $C^d$ onto the affine hyperplane containing that face--by symmetry, the constant depends only on $1/2$ (and, of course, on $d$) but not on the face of $C^d$; we show that $\pi$ is bilipschitz in Section~\ref{subsubsecHaCoC}.  For balls in this game, we use only the restriction to the subspace $\partial X$ of the balls in $\RR^{d+1}$ that are centered at a point in $\partial X$ and with radius length given by $\| \cdot\|_2$.\footnote{Since $\partial X$ is an affine variety (hence closed) intersected with the unit cube in $\RR^{d+1}$ (also closed), it is a complete metric subspace of $\RR^{d+1}$, and thus we may play the game.  Moreover, $\partial X$ is a smooth $d-1$ manifold because it is a $d$-dimensional light-cone intersected with the unit $d$-cube and fixing a face of the cube means substituting $\pm 1$ into the corresponding variable in the light-cone, which yields a $d-1$-dimensional level-surface. Since, for any quadratic variety, all points different from the origin are nonsingular, $\partial X$ has no singular points and is thus a smooth manifold (possibly with boundary since a face of the cube is a manifold with boundary) or, possibly, the empty set since a face of the cube may miss the light-cone--but, of course, some face must meet the light-cone.}

Define a subset of $X_m(\ZZ)$ as follows:  \[P_0 := \bigcup_{ \|U\|_{q_2} < 3\sqrt{|m|}} X_{\|U\|_{q_2}} \bigg\backslash \{\langle 0, 0 \rangle\}.\]  By (\ref{eqnnormrelation}), we surmise that $P_0$ is contained in a large enough ball and thus a finite set.  Normalizing each point of $P_0$ by dividing by its sup norm yields a unique minimal positive distance $d_0$ (depending only on $Q$ and $m$ and with respect to $\|\cdot\|_2$) between these normalized points. 


Moreover, since $\partial X$ is a compact, isometrically embedded Riemannian submanifold (under inclusion) of $\RR^{d+1}$ with Riemannian metric induced by the usual dot product on $\RR^{d+1}$, it has a finite number of path components, each with some diameter (with respect to $\|\cdot\|_2$);\footnote{The notions of normality, orthogonality, and angle in this proof are all with respect to this dot product.}  and, therefore, $\partial X$ must meet the boundary of any closed $\|\cdot\|_2$-ball around any point of $\partial X$ with diameter less than the least diameter--denote this $d_1$--of the path components.  Note that $d_1 >0$ because $d\geq2$.  Since there are only a finite number of path components (and these are closed sets of $\RR^{d+1}$), there exists a unique minimal positive distance $d_2$ (with respect to $\|\cdot\|_2$) between any two components.

To play the strong game, Player $A$ is allowed to pick balls with radii greater than or equal to $\a$ times the radius of Player $B$'s most recent choice of ball.  For this proof, we agree that Player $A$ always chooses a ball with radius equal to $\a$ times the radius of Player $B$'s most recent choice of ball.  Therefore, after iterating the game a finite number of times, we can force Player $B$'s balls to have arbitrarily small radii.  Fix a very small $\varepsilon >0$ and  let $R > 0$ be as in Lemma~\ref{lemmClosetoLinear}.\footnote{The smaller the $\varepsilon$, the larger the constant $\a$ that we could have started with--however, the current proof does not allow the maximal value of $1/2$ for $\a$.  To obtain this maximal value of $\a$, one should be able to use Schmidt's original technique in~\cite{Sch2}.  See the Conclusion for a more detailed remark.}  Iterate the game so that Player $B$'s balls have radii strictly smaller than $R_0:=\frac 1 3 \min((6 \kappa_0 \sqrt{|m|})^{-1}, d_0/2, d_1/2, d_2/2,1/2, c_\pi /\sqrt{d}, R)$. 

Player $B$ begins by picking a closed ball $B_1$ with $c(B_1) \in \partial X$ and $\rho(B_1)<R_0$.  Now $B_1$ could meet more than one face of $C^d$.  It is, however, more convenient to play the game on a ``piece'' of $d$-dimensional (affine) hyperplane and then project onto the cube.  To do this, pick a face $\mF$ that contains $c(B_1)$;\footnote{If there is a choice of face, pick any one of them.} this face determines a $d$-dimensional (affine) hyperplane $\mE$.  Thicken this hyperplane by $1/2$; intersect the thickening with $C^d$; and denote this intersection by $\mF'$.  Note that, for any closed ball $B'$ centered in $\mF$ with radius at most $\rho(B_1)$, one has $B' \cap \mF' = B' \cap C^d$.

\subsubsection{Handling a corner of $C^d$}\label{subsubsecHaCoC}

Now, for every $x \in \mF'$, there exists a unique ray emanating from the origin $\boldsymbol {0}$ (of $\RR^{d+1}$) that intersects $\mE$ in a unique point, which we denote $\pi(x)$.  Whence we have the radial projection $\pi: \mF' \rightarrow \mE$, which is the identity on $\mF$ and, in general, a bilipschitz homeomorphism onto its image.\footnote{Distance in both the domain and range are inherited from $\RR^{d+1}$.}  To see the later property, first note that $\pi$, by definition, is bijective onto its image.  We show that $\pi$ and its inverse are Lipschitz; consider $\pi^{-1}(v) \not\sim \pi^{-1}(w) \in \mF'$.  These points also lie on the unit cube.  Now $1\leq \|\pi(\mF')\|$ is bounded from above by some positive number $M$ because $1/2$ is small enough and because, for any $x \in \mF' \backslash \mF$, one can consider the projection in the $2$-plane determined by $x$ (thought of as a vector in $\RR^{d+1}$) and the normal vector of $\mE$.\footnote{This $2$-plane intersects $\mE$ in a line, which must be normal to the normal vector; thus we obtain a right triangle in this $2$-plane.  Since $1/2$ is considerably smaller than $1$, the angle between $x$ and the unit normal vector of $\mE$ (whose initial point is $\boldsymbol{0}$ and terminal point lies on $\mE$) is bounded away from being orthogonal, and thus the length of the hypothenuse is bounded.} Therefore, there are numbers $1 \leq c_v, c_w \leq M$ such that $\pi(\pi^{-1}(v)) = c_v \pi^{-1}(v)$ and $\pi(\pi^{-1}(w)) = c_w \pi^{-1}(w)$.  By  Lemma~\ref{lemmCubeMissesLine}, we have  \[\|\pi^{-1}(v) -  \frac {c_w}{c_v} \pi^{-1}(w) \| > c' \|\pi^{-1}(v)-\pi^{-1}(w) \|\] where $c'$ is a constant.  This shows that $\pi^{-1}$ is Lipschitz.  

To show that $\pi$ is Lipschitz, we rename the variables so that $\mE$ has equation $Y_1 = 1$.  Consider the following sup-like norm on $\RR^{d+1}$:  $\|v\|_s := \max(|v_1|, \frac 1 {M'} |v_2|, \cdots, \frac 1 {M'} |v_{d+1}|)$ where the $v_i$s are the components of $v$ and $M'$ is a positive constant larger than $\sup \|\pi(\mF')\|_2$.  Let $S:=\{v \in \RR^{d+1} \mid \|v\|_s = 1\}$, the $d$-dimensional shell of a $d+1$-dimensional box in $\RR^{d+1}$.  Hence, the face of $S$ normal to (and containing the terminal point of) the standard basis vector $e_1:=(1, 0, \cdots, 0)$ contains $\pi(\mF')$.  Thus, given $x \not\sim y \in \mF'$, we have that $\pi(x), \pi(y) \in S$.  Using Lemma~\ref{lemmCubeMissesLine} with $\|\cdot\|_s$  in the analogous way as for the $\pi^{-1}$ case shows that $\pi$ is bilipschitz.  Now, by symmetry, $\pi$ is well-defined and bilipschitz for any face of $C^d$.  Moreover, the bilipschitz constant $c_\pi\geq1$ (with respect to distance given by $\|\cdot\|_2$) depends on our choice of $1/2$ and, in particular, is independent of the face of $C^d$, the light-cone, and the Schmidt game; thus it is a universal constant.

Now we must show that $\pi(\mF')$ contains a big enough region of $\mE$:  precisely, we show that $\pi(\mF')$ contains the set $\tilde{\mF}:=\{v \in \mE \mid \|v\|\leq 1 + 1/\sqrt{d}\}$ where we continue to assume that $e_1$ is the outward normal of $\mE$ and the origin of $\mE$ is the terminal point of $e_1$.\footnote{Here $\|\cdot\|$ denotes the sup norm of $\mE$, not the sup norm of $\RR^{d+1}$.}  Let $x \in \partial \mF$.  If $y \in \mF$, then the segment between $x$ and $y$ is in $\pi(\mF')$ because $\mF$ is convex.  If $y$ belongs to a face $\mF''$ (of $C^d$) adjacent to $\mF$ containing $x$, then these points $x$, $y$, and $\boldsymbol {0}$ determine a $2$-plane $\mP$ in $\RR^{d+1}$.  Now $\mP$ intersects $\mE$ and $\mF''$ in lines.  Consider $z: = t\pi(x) + (1-t)\pi(y)$ for some $t \in [0,1]$.  Since it lies on the intersection line with $\mE$, it must lie in $\mP$.  Therefore, some multiple of it $z''$ lies on the intersection line in $\mF''$.  Since $x$ and $y$ lie on the same face of $C^d$, they cannot lie on the same lie through $\boldsymbol {0}$.  Consequently, the angle in $\mP$ between $x$ and $y$ is strictly smaller than a straight angle and, moreover, it is bisected by $z''$.  Thus the intersection of $z$ (thought of as a vector in $\RR^{d+1}$) with $\mF''$ must lie between $x$ and $y$.

Let us continue to assume that $e_1$ is the outward normal vector of $\mE$.  Now pick a vector $v$ normal to $e_1$.  Then the ray $R_v$ determined by $v$ and starting at the terminal point of $e_1$ lies in $\mF$ and intersects $\partial \mF$ at a point $p$.  Let $\mP$ be the $2$-plane determined by $e_1$ and $v$.  Now the segment starting at $p$ and ending at (the terminal point of) $p-\frac 1 2 e_1$ is in $\mF'$. Therefore, similar triangles in $\mP$ implies $\pi(p-\frac 1 2 e_1)$ lies on $R_v$ outside of $\mF$ with a distance (with respect to $\|\cdot\|_2$) of at least $1$ from $p$.  Since $p$ lies on $\partial \mF$, it must lie on, at least, another face of $C^d$--thus, some other coordinate besides the first must have value $1$ or $-1$.  Looking at the coordinates of $p$ (which has value $1$ in the first coordinate), we note that adding $-\frac 1 2 e_1$ is in this adjacent face.  Thus, by the proceeding paragraph, every point on the segment from $\pi(p-\frac 1 2 e_1)$ to (the terminal point of) $e_1$ lies in $\pi(\mF')$.  See Figure~\ref{figproj1}.

\begin{figure}[h] 
   \centering
   \includegraphics[width=1.0\textwidth]{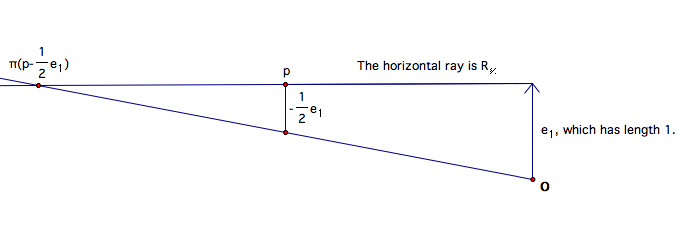} 
    \caption{On a $2$-plane in $\RR^{d+1}$ containing the normal vector to face $\mF$.}
   \label{figproj1}
\end{figure}

Now, in this paragraph, we restrict consideration solely to $\mE$.  Let $x \in \tilde{\mF} \backslash \mF$.  Then $x$ determines a unique ray from the origin $\boldsymbol {0}_E$ of $\mE$ (i.e. the terminal point of $e_1$), which intersects a $d-1$-dimensional face $\mS$ of the boundary of $\mF$ in a point $p$.  The largest distance (with respect to $\|\cdot\|_2$) that $p$ can be from $\boldsymbol {0}_E$ is $\sqrt{d}$.  Let $N_p$ denote the normal line (in $\mE$) to $\mS$.  The vector $p$ and the line $N_p$ determine a $2$-plane $\mP$ in $\mE$.  Similar triangles in $\mP$ implies that if we thicken $\mS$ by any length less than $1/\sqrt{d}$, we do not meet $\pi(p-\frac 1 2 e_1)$.  If $p$ lies on more than one face of $\mF$, then the same calculation can be made.  Therefore, $x \in \pi(\mF')$.  See Figure~\ref{figproj2}.

\begin{figure}[h] 
   \centering
   \includegraphics[width=1.0\textwidth]{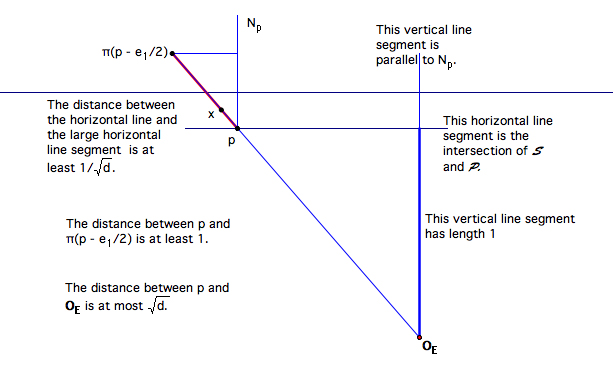} 
    \caption{On a $2$-plane in $\mE$.}
   \label{figproj2}
\end{figure}

Finally, since the light-cone consists of lines through the origin $\boldsymbol {0}$ of $\RR^{d+1}$, the restriction of $\pi$ to the light-cone is a well-defined bilipschitz homeomorphism.  

Now pick any ball $B \subset \RR^{d+1}$ centered at some point $ p \in \mF$ with radius $r \leq 1/2$.  If $B$ meets some other face $\mF''$ of $C^d$, then let $q$ denote any point in the intersection.  The shortest distance between $q$ and $\mF$ is given by the distance along (the direction of) the normal line to $\mF$; since $p$ lies in $\mF$, this normal distance is $\leq r$; thus $B \cap C^d = B \cap \mF'$.  Now $\pi$ preserves $p$ and, therefore, there exists a unique ($\|\cdot\|_2$-ball) $B^\pi \subset \mE$ with the same center $p$ and radius $c^{-1}_\pi r$ contained in $\pi(B \cap \mF')$.  Forcing $r \leq \min(1/2, c_\pi/ \sqrt{d})$ implies that $B^\pi \subset \tilde{\mF}$.  Note that we have required this for Player $B$'s choice of $B_1$ (even more, the ball of $3$ times the radius meets this requirement too).  Also, note that the game is local in the following sense:  once $B_1$ is chosen and $\mF$ is chosen, then $\mF'$ is fixed because any later balls lie in $B_1$.  Therefore, for balls of $\RR^{d+1}$ inside of $B_1$, we can loosen the definition of $B^\pi$ to the ($\|\cdot\|_2$-ball) contained in $\mE$ with center $\pi(p)$ and radius $c^{-1}_\pi r$ contained in $\pi(B \cap \mF')$ and, thereby, allowing the center of $B$ to lie near, but not necessarily on, $\mF$. 

Now consider the inverse:  pick any ball $B \subset \tilde{\mF}$ centered at some point $\tilde{p} \in \tilde{\mF}$ with radius $r$. Now there exists a unique ($\|\cdot\|_2$-ball) $B^{-\pi} \subset \RR^{d+1}$ with center $\pi^{-1}(\tilde{p}) \in \mF'$ and radius $c^{-1}_\pi r$ such that $B^{-\pi} \cap \mF'$ is contained in $\pi^{-1}(B)$.  

Fix this face $\mF$, which contains $c(B_1)$; let $\tilde{\mF}$ be as above.  Now $\tilde{\mF}$ intersect the light-cone (denote it by $\tilde{\partial} X$) is a hypersurface of the $d$-dimensional Euclidean space $\mE$ contained in $\tilde{\mF}$ (the terminal point of $e_1$ is thought of as the origin of $\mE$).  Therefore, Lemma~\ref{lemmClosetoLinear} applies and all radii are smaller than the $R > 0$ from this lemma as stated in the beginning. 

\subsubsection{Playing the game}

Player $B$ has already chosen $B_1$.  Form $B_1^\pi \subset \tilde{\mF}$.  Note that $\rho(B_1)$ is so small that $B_1$ (or even the ball with $3$ times the radius) meets only one path-component of the variety and at most one normalized point  of $P_0$, a point that we denote by $q_0$.  (Recall that we normalize by dividing by the sup norm, so $q_0 \in C^d$.)  Let $q:=\pi(q_0)$.

\paragraph{\it{Missing a line}}\label{secMissLine}


Let us assume that $d \geq 3$.\footnote{For $d=2$, as we shall see below, missing the line that we need to miss is equivalent to missing a point, and the latter only requires a simplified form of the proof in this section.}  We now restrict to $\mE$ as our ambient space.  Since, in general, we need to miss not just a point but a line, we show how to miss a line $L$ containing $q$ and lying in $\mE$.  Let $p:=c(B_1^\pi)$, and let $L_p$ denote the line in $\mE$ parallel to $L$ and meeting $p$.  Let $\mN_p$ be the normal line of $\tilde{\partial} X$ at $p$ in $\mE$.  Now $T_p (\tilde{\partial} X) \cap L_p$ is either $p$ or $L_p$.  If this intersection is $L_p$, then let $L'_p := L_p$.  Otherwise, if this intersection is just $p$, then project, along the $\mN_p$ direction, the line $L_p$ onto $T_p (\tilde{\partial} X)$;\footnote{The line $L_p$ has equation (in $\mE$) $p + t u$ where $t \in \RR$ and $u$ is its direction vector.  Projection is the relevant addition of some multiple $s \in \RR$ of the direction vector $n$ of $\mN_p$ to $u$ so that the resulting line $p + t(u + s n)$ lies in $T_p (\tilde{\partial} X)$.} denote the projected line by $L'_p$.  Note that $L'_p$ contains $p$.    Let us first assume that $L'_p$ is not just the point $p$, but a proper line in the tangent space--this implies that the direction vector of $L$ and $L_p$ is not along the $\mN_p$ direction.  Whence, by Gram-Schmidt, there is an element $v$ of $T_p (\tilde{\partial} X)$ normal to $L'_p$.  
 Then $v$ and $\mN_p$ determine a $2$-plane $\mP_\mN$.  Also, let $\mP'_p$ denote the $2$-plane determined by $L_p$ and $\mN_p$; note that $\mP'_p$ contains $L'_p$.  And let $\mP'_q$ denote the $2$-plane spanned by $L$ and the line parallel to $\mN_p$ meeting $q$; note that $\mP'_p$ and $\mP'_q$ are translated (from $p$ to $q$) $2$-planes and thus parallel or identical.  Consequently, $v$ is normal to these planes.

If $L$ and $L_p$ do not coincide, then they form a $2$-plane $\mP_L$.\footnote{This plane is determined by $L_p$ and the segment between $p$ and $q$. In particular, if we regard, for the moment, $p$ as the origin of $\mE$ and that the line $L_p$ has direction vector $u$, then $L$ is explicitly described as $q + tu$ for $t \in \RR$; therefore, the vectors $q$ and $u$ span $\mP_L$.}  Now since $\mP_\mN$ and $\mP_L$ contain $p$, they determine (at most) a $4$-space $\mS$.  Since $\mS$ contains $p$ and $q$ and the direction vectors of $L$ and $\mN_p$, both $\mP'_p$ and $\mP'_q$ are contained in $\mS$, and, since $\mP'_p$ and $\mP'_q$ are still parallel (or identical) in $\mS$, there is some vector $w$ with initial point $p$ and terminal point on $\mP'_q$ of least distance and lying in $\mS$.    If $w$ is not the zero vector, then either $v$ or $-v$ has angle greater than or equal to orthogonal with respect to $w$; without loss of generality, we may assume that $v$ does.  By Lemmas~\ref{lemmClosetoLinear} and~\ref{lemm2PlaneIntH}, the point $p'$ of the variety on the boundary of $B_1^\pi$ inside $\mP_\mN$ in the direction $v$ is very close to $T_p(\tilde{\partial} X)$.\footnote{The plane $\mP_\mN$ meets the tangent space of the variety in a line, which must be determined by $v$. Now the two distinct points from Lemma~\ref{lemm2PlaneIntH} must correspond to going in directions $v$ and $-v$, respectively.}  Therefore, it is far away from the part of $T_p(\tilde{\partial} X)$ that is closest to $L'_p$ and even farther away (the distance is with respect to $\|\cdot\|_2$) from $L$.  Moreover, since $p'$ is far away from $L$ in $\mS$, it is far away from $L$ in $\mE$, as $L \subset \mS \subset \mE$.  

If $w$ is the zero vector (equivalently, $\mP'_p$ and $\mP'_q$ are identical), then it does not matter whether $v$ or $-v$ is chosen (as both move orthogonally away from $L$, $L_p$, and $L'_p$) and, in the same way as in the previous paragraph, $p'$ is far away from the part of $T_p(\tilde{\partial} X)$ that is closest to $L$.   If $L$ and $L_p$ coincide, then $\mP'_p$ and $\mP'_q$ are identical, and the previous sentence applies.

For the other case, when $L'_p$ is just the point $p$, we can pick any $v \in T_p (\tilde{\partial} X)$ (all such vectors are normal to $L'_p$) and repeat the last two paragraphs with $L_p$ replacing $\mP'_p$ and $L$ replacing $\mP'_q$ (note that the dimension of $\mS$ is now at most $3$).

Pick a point of the variety in $\mP_\mN$ near $p'$ as the center of Player $A$'s ball, which we denote by $\tilde{A}_1$, such that $\tilde{A}_1 \subset B_1^\pi$ and has radius $c_\pi \a \rho(B_1)$.  Since $\a$ is small enough (if we have chosen $\varepsilon$ very small, then $c^2_\pi \a$ need only be slightly smaller than $1/2$), $\tilde{A}_1$ will not meet $L$.\footnote{A more explicit computation is in (\ref{eqnfittingcalc}).}

Let $A_1 := \tilde{A}_1^{-\pi}$.  Then $A_1$ does not meet $\pi^{-1}{L}$ and, in particular, $q_0$. Note that $c(A_1) \in \partial X$ and $\rho(A_1) = \a \rho(B_1)$.

Finally, Player $B$ will choose another ball inside of $A_1$, which, by reindexing, we may assume is the first ball $B_1$.\footnote{This new $B_1$ may not be centered in $\mF$.  In this case, we can either rename $\mF$ and its related objects or just ignore the distinction as it does not matter for the rest of the proof.}  Therefore, without loss of generality, we may assume that $B_1$ does not meet any normalized point from $P_0$.

\paragraph{\it{Essence of the proof}}\label{secEssenseProofLW}
Player $B$ has chosen $B_1$.  We play for Player $A$ using induction on the iterations of the game, iterations that are denoted by $i$.

Define the constant $c' := \min(\frac {(\a\b)^2}{4\kappa_0}, \frac {(\a\b)^2}{4}, \frac {3 \a\b\rho(B_1)\sqrt{|m|}}{4 \kappa_0}, \frac {3 \a\b\rho(B_1)\sqrt{|m|}}{4})$.  For $i \in \NN$, let $M_i := \lceil \rho(B_i)^{-1}\rceil$.  We window elements of $X_m(\ZZ)$ as follows:  \[P_1 := \bigcup_{3\sqrt{|m|} \leq \|U\|_{q_2} \leq \frac{M_1}{\kappa_0}} X_{\|U\|_{q_2}}\] and, for natural numbers $i > 1$, \[P_i := \bigcup_{\frac{M_{i-1}}{\kappa_0} \leq \|U\|_{q_2} \leq \frac{M_i}{\kappa_0}} X_{\|U\|_{q_2}}.\]  We delay considering approximation by elements of the finite set $P_0$ until the end.

The constant $c'$ is scaled correctly:

\begin{lemm}\label{lemmcprimecorrect}  
\begin{align*}
\textrm{We have} \quad \frac{c'}{3 \sqrt{|m|}} \leq \frac 1 4 (\a\b)\rho(B_1) & \quad \textrm{when }  \quad i = 1, \\ 
\textrm{and} \quad \frac{c' \kappa_0}{M_{i-1}} \leq \frac 1 4 (\a\b)\rho(B_i) &  \quad \textrm{when }  \quad i > 1.
\end{align*}
\end{lemm}

\begin{proof}
The $i=1$ case follows by the definition of $c'$.  For $i >1$, note that \[\frac{c' \kappa_0}{M_{i-1}} \leq \frac 1 4 (\a\b)^2\rho(B_{i-1})\leq  \frac 1 4 (\a\b)\rho(B_{i}).\]
\end{proof}

\subparagraph*{\it{Initial step $i=1$}} 

We first consider $d \geq 3$.  Let $BB_1$ denote the ball of $\RR^{d+1}$ containing $B_1$ with the same center, but twice the radius; and $BBB_1$ the same, but with triple the radius.\footnote{We use the analogous notation to mean the same for other balls.}  If no normalized (which, recall means that we divide by its sup norm) point of $P_1$ is in $BB_1$, then Player $A$ may freely choose any allowed ball--for definiteness, let $c(A_1) = c(B_1)$ and $\rho(A_1) = \a \rho(B_1)$.  Note that, by Lemma~\ref{lemmcprimecorrect}, $AAA_1$ misses the ball of radius $c'/6 \sqrt{|m|}$ around any normalized point of $P_1$.    

If at least one normalized point of $P_1$ is in $BB_1$, then we proceed as follows.  By Proposition~\ref{propIntegerSepartionLW}, if $BB_1$, which has small enough diameter, meets any two distinct normalized points $\frac{\langle w, u \rangle}{\|\langle w, u \rangle\|},\frac{\langle w', u' \rangle}{\|\langle w', u' \rangle\|}$ of $P_1$, then $\langle w, u \rangle \approx \langle w', u' \rangle$.  Thus $\langle w, u \rangle$ and $\langle w', u' \rangle$ lie on the same $2$-plane $\mP$ through the origin $\boldsymbol{0}$ of $\RR^{d+1}$.\footnote{If there is exactly one normalized point of $P_1$ in $BB_1$, we can, at random, pick such a $2$-plane $\mP$ and continue to follow this proof, or note that it is easier to miss just a point and use the analogous proof for the light-cone case, namely Sections~\ref{secMissPoint} and~\ref{secEssProofLC}}  Consequently, any normalized point of $P_1$ in $BB_1$ must thus lie on $\mP$.  Since $\mP$ contains $\boldsymbol{0}$, it cannot coincide with the affine hyperplane $\mE$.  Now since $BB_1$ contains a normalized point of $P_1$ and its radius is small enough, the normalized point projects (under $\pi$) onto $\mE$.  Therefore, $\mE \cap \mP$ is a line $L$.  Player $A$ must miss $L$ and can do so using the technique in Section~\ref{secMissLine}.

And even more, Player $A$ must miss a neighborhood of $L$, namely the set \[L' := \bigcup_{q \in L} B\big(q, \frac{c_\pi c'}{6 \sqrt{|m|}}\big),\] where the union is over $\|\cdot\|_2$-balls in $\mE$.  Let us use the notation from Section~\ref{secMissLine}:  $L$ and $L_p$, where $p$ is the center of $B_1$.  Since, in that section, we moved away in a normal direction to $L$ and $L_p$ and are far away from $L_p$ and even farther away from $L$, the restriction on Player $A$'s choice of ball is given by fitting it between the balls $B(p, \frac{c'}{6 \sqrt{|m|}})$ and $B_1$ in $\RR^{d+1}$.  By Lemma~\ref{lemmcprimecorrect}, we have (after projecting onto $\mE$)\footnote{Since we are on the variety, we must use Lemma~\ref{lemmClosetoLinear}.  And thus the $\varepsilon$ factor appears, but is chosen very small so that it does not affect this restriction--note that, for this proof, we need only $6 \a c_\pi \rho(B_1)$ amount of room, but we have much more.}  \begin{eqnarray}\label{eqnfittingcalc}c_\pi^{-1}\rho(B_1) -\frac{c_\pi c'}{6 \sqrt{|m|}} \geq c^{-1}_\pi \rho(B_1) - \frac{c_\pi \a}{8} \rho(B_1) > 7 \a c_\pi \rho(B_1).\end{eqnarray}  Recall the definition of $\mP_\mN$ from Section~\ref{secMissLine}.  In $\mP_\mN$, there is an arcsegment contained in $\tilde{\partial} X$ connecting $p$ with the correct point $p'$ of the boundary of $B^\pi_1$.  This arcsegment is a smooth curve and thus continuous.  In particular, there is some point on $\tilde{\partial} X$ at all distances from $p$ to $p'$ where distance is with respect to $\|\cdot\|_2$ in $\mE$.  Consequently, there is a choice for $\tilde{A}_1$ such that $\widetilde{AAA}_1$ does not meet $L'$.  Let $A_1 := \tilde{A}_1^{-\pi}$.\footnote{Since $\widetilde{AAA}_1 \subset B_1^\pi$, we have that $AAA_1 \subset B_1$.}  Then $AAA_1$ does not meet $\pi^{-1}{L'}$, which is a big enough set to contain the intersection of $C^d$ and the balls of radius $\frac{c'}{6 \sqrt{|m|}}$ around all points of $\pi^{-1}(L)$. (Note that $\rho(A_1) = \a \rho(B_1)$, as in Section~\ref{secMissLine}.)

In particular, regardless of whether a normalized point of $P_1$ is in $BB_1$ or not, we have shown that all points of $AAA_1 \cap C^d$ are outside of \[\bigcup_{\langle w, u \rangle \in P_1} B\big(\frac{\langle w, u \rangle}{\|\langle w, u \rangle\|}, \frac{c'}{2\|u\|_{q_2}} \big) \cap C^d.\] By (\ref{eqnequivnorms2}) and (\ref{eqnuApproxBnds}), there exists some constant $c_{2s}$ depending only on $c_s$ such that all points of $AAA_1 \cap C^d$ are outside of \[\bigcup_{\langle w, u \rangle \in P_1} B\big(\frac{\langle w, u \rangle}{\|\langle w, u \rangle\|}, \frac{c'c_{2s}}{2\|\langle w, u\rangle\|} \big) \cap C^d.\]

We are still considering the initial step $i=1$, but now we consider the case $d=2$; we need only adapt the $d\geq3$ proof.  Using $x, y,$ and $z$ as the variables in $\RR^3$ and renaming them if necessary, we may, without loss of generality, assume that the light-cone has equation $ax^2 + by^2 = cz^2$ for positive, rational coefficients $a, b,$ and $c$.  

If no normalized point of $P_1$ is in $BB_1$, then proceed as in the $d\geq3$ case above.  If at least one normalized point of $P_1$ is in $BB_1$, then, proceeding as in the $d\geq3$ case, we obtain the $2$-plane $\mP$.  Let us follow the usual convention and call the $z = \pm1$ faces of $C^2$ the horizontal faces and the other faces the vertical faces; let $e_1$, $e_2$, and $e_3$ be the unit $x$, $y$, and $z$-vectors, respectively.  Note that $\mP$ must contain the $z$-axis.  Therefore, the intersection of $\mP$ with a horizontal face of $C^2$ is a line $L$ containing the terminal point of either $e_3$ or $-e_3$.  By plugging in $\pm 1$ into the $z$-variable, we see that the intersection of the light-cone with a horizontal face is an ellipse $E$ around (but not containing) the terminal point of either $e_3$ or $-e_3$.\footnote{Since we are using the projection $\pi$ and Lemma~\ref{lemmClosetoLinear}, it does not matter if these actual faces of $C^2$ contain all of their ellipses--a slightly enlarged face (obtained via $\pi$) will contain enough of the ellipse for this proof.}  By Lemma~\ref{lemmClosetoLinear} and the fact that $\varepsilon$ is very small, $L \cap E$ can meet $BB_1$ in at most one point.  Missing a point is easier than missing a line--in particular there are only two directions ($v$ or $-v$) in the tangent line to go, so there is no need to appeal to Gram-Schmidt when using Section~\ref{secMissLine}.\footnote{For more details, see the simplification of Section~\ref{secMissLine} in Section~\ref{secMissPoint}.}

By plugging in $\pm 1$ into the other variables (one at a time), we see that the intersection of the light-cone with the vertical faces are hyperbolas with such orientation that any line parallel to the $z$-axis (and in the vertical face) meets a connected component of the hyperbola in exactly one point.\footnote{Again, it does not matter if these actual faces of $C^2$ contain all of their hyperbolas.}  Note that the intersection of $\mP$ and any vertical face is a line parallel to the $z$-axis.  Recall that we have required $BB_1$ to meet only one connected component.  Therefore, we need only miss this intersection point, as in the previous paragraph.

Since the rest of the proof is analogous (the same relation (\ref{eqnfittingcalc}) holds and $\mP_\mN$ is just the affine $2$-plane containing the relevant face of $C^2$) to the $d\geq3$ case, Player $A$ can pick $A_1$ such that all points of $AAA_1 \cap C^d$ are outside of \[\bigcup_{\langle w, u \rangle \in P_1} B\big(\frac{\langle w, u \rangle}{\|\langle w, u \rangle\|}, \frac{c'c_{2s}}{2\|\langle w, u\rangle\|} \big) \cap C^d.\]

\subparagraph*{\it{Induction step.}}

Assume that all points of $AAA_{i-1} \cap C^d$ are outside of \[\bigcup_{\langle w, u \rangle \in \cup_{j=1}^{i-1}P_j} B\big(\frac{\langle w, u \rangle}{\|\langle w, u \rangle\|}, \frac{c'c_{2s}}{2\|\langle w, u\rangle\|} \big) \cap C^d.\]  Now Player $B$ may freely (but according to the rules of the strong Schmidt game) choose $B_i \subset A_{i-1}$ (once this is done, $M_i$ is determined and so is $P_i$).  Now every point of $BB_i$ is contained in $AAA_{i-1}$.\footnote{Every point of $BB_i$ is within $2\rho(A_{i-1})$ of $c(B_i)$ and $c(B_i)$ is within $\rho(A_{i-1})$ of $c(A_{i-1})$; these two facts show the assertion.}  Thus, all points of $BB_i \cap C^d$ are outside of \[\bigcup_{\langle w, u \rangle \in \cup_{j=1}^{i-1}P_j} B\big(\frac{\langle w, u \rangle}{\|\langle w, u \rangle\|}, \frac{c'c_{2s}}{2\|\langle w, u\rangle\|} \big) \cap C^d.\]

Now the proof of the induction step is the same as the initial step, except $i$ replaces $1$ and $M_{i-1}/\kappa_0$ replaces $3 \sqrt{|m|}$ everywhere.  Thus, we may conclude that all points of $AAA_{i} \cap C^d$ are outside of \[\bigcup_{\langle w, u \rangle \in \cup_{j=1}^{i}P_j} B\big(\frac{\langle w, u \rangle}{\|\langle w, u \rangle\|}, \frac{c'c_{2s}}{2\|\langle w, u\rangle\|} \big) \cap C^d.\]

Thus, all points of $BB_{i+1} \cap C^d$ are outside of \[\bigcup_{\langle w, u \rangle \in \cup_{j=1}^{i}P_j} B\big(\frac{\langle w, u \rangle}{\|\langle w, u \rangle\|}, \frac{c'c_{2s}}{2\|\langle w, u\rangle\|} \big) \cap C^d\] for all $i \in \NN$.


\subparagraph*{\it{Finishing the proof}}

Let $\tilde{P}_0$ denote the set of normalized points of $P_0$.  Changing norms from $\|\cdot\|_2$ to $\|\cdot\|$ using the constant $c'_{2s}>0$, we have shown that the following is an $\a$-strongly winning set:

\[BA':=\big{\{}v \in \partial X \backslash \tilde{P}_0 \mid \big{\|}\frac{\langle w, u \rangle}{\|\langle w, u \rangle\|} - v\big{\|} > \frac{c'c_{2s}c'_{2s}}{2\|\langle w, u\rangle\|} \textrm{ for all } \langle w, u \rangle \in \cup_{i=1}^\infty P_i\big{\}}.\]

Let $v \in BA'$.  Since $P_0$ is a finite set, there exists a unique minimal positive distance (with respect to $\|\cdot\|$) between $v$ and the normalized points of $P_0$.  Thus, one can shrink the constant to some $c(v)>0$ such that $c \leq \frac{c'c_{2s}c'_{2s}}{2}$ and \[\big{\|}\frac{\langle w, u \rangle}{\|\langle w, u \rangle\|} - v\big{\|} \geq \frac{c}{\|\langle w, u\rangle\|} \textrm{ for all } \langle w, u \rangle \in X_m(\ZZ)\backslash \{\langle 0, 0 \rangle\}\big{\}},\] thereby implying that the $\a$-strongly winning set $BA'$ is $BA_{\partial X}^\psi(X_m(\ZZ))$.

\begin{rema}
Since we always choose $A_i$ such that $\rho(A_i) = \a \rho(B_i)$, our proof shows both $\a$-strong winning and $\a$-winning.
\end{rema}

\begin{rema}\label{rmkGeneralQuadForm}
An arbitrary (rational, nondegenerate, indefinite) quadratic form $\tilde{Q}$ is equivalent to some diagonal (rational, nondegenerate, indefinite) quadratic form $Q$ (see Corollary~7.30 of~\cite{EKM}); hence there exists a matrix $M \in GL_{d+1}(\QQ)$ such that $Q (\cdot) = \tilde{Q} (M \cdot)$.  The proof for $\tilde{Q}$ is virtually the same as the proof for $Q$.  There are two versions of this proof; we give one here and leave the other one to the Conclusion.  The main change is that the vectors $\langle w, u \rangle$ in $\ZZ^{d+1}$ for the diagonal form $Q$ are now in $M^{-1} \ZZ^{d+1}$ for the arbitrary form.  We use the integral property of $\langle w, u \rangle$ in the proofs of Lemmas~\ref{lemIntegerSepartionLW} and~\ref{lemIntegerSepartionLC}; however, those proofs remain unchanged for arbitrary forms except that the constant $\kappa_0$ is multiplied by the square of the entry of $M$ with the largest denominator in absolute value--note that this denominator is an integer different from $0$.  Since $M$ depends only on the form, $\kappa_0$ is still a constant that depends on the form.  Now our proof above actually shows that approximation by, not integer lattice points (i.e. integer vectors), but by (the relevant) elements of $M^{-1} \ZZ^{d+1}$ results in an $\a$-strong winning and $\a$-winning set; let us denote this set by $BA''$.   A vector $v \in BA''$ satisfies \[\big{\|}\frac{\langle w, u \rangle}{\|\langle w, u \rangle\|} - v\big{\|} \geq \frac{c''}{\|\langle w, u\rangle\|}\] for some constant $c''(v)>0$ and all relevant elements $\langle w, u \rangle$ of $M^{-1} \ZZ^{d+1}$.  By applying Lemma~\ref{lemmCubeMissesLine}, we have that \[\big{\|}\frac{\langle w, u \rangle}{\|M\langle w, u \rangle\|} - \frac{v}{\|Mv\|}\big{\|} \geq \frac{c'''}{\|M\langle w, u\rangle\|}\] for some constant $c'''(v)>0$.  Since $\|M \cdot\|$ is a norm and all norms are equivalent on $\RR^{d+1}$, we have \[\big{\|}\frac{M\langle w, u \rangle}{\|M \langle w, u \rangle\|} - \frac{Mv}{\|Mv\|}\big{\|} \geq \frac{c}{\|M\langle w, u\rangle\|}\] for some constant $c(Mv)>0$.  This shows that $\fp(M(BA''))$ is badly approximable in the desired way.\footnote{Note that $v \in \partial X := \{v \in M^{-1} \RR^{d+1} \mid \tilde{Q}(M v) =0\} \cap C^d$ and, since $Mv$ satisfies $\tilde{Q} = 0$, so does $Mv/\|Mv\|$.  In other words, if we call $\partial X$ the boundary variety, then there are two boundary varieties here:  one for $\tilde{Q}$ and one for its equivalent diagonal form--the bijection $\fp \circ M$ restricts to a bijection (actually, a bilipschitz homeomorphism, as we shall see) of these boundary varieties.}  Now the map $\fp \circ M: C^d \rightarrow C^d$ is Lipschitz because $\|M \cdot \|$ is a norm, because any vector of $M(C^d)$ has $\|\cdot\|$-norm bounded between two universal positive constants, and because Lemma~\ref{lemmCubeMissesLine} applies.  But the inverse map is $\fp \circ M^{-1}: C^d \rightarrow C^d$, and thus the map $\fp \circ M: C^d \rightarrow C^d$ is bilipschitz (with constant $c_{M}\geq 1$ depending only on the arbitrary quadratic form).  Consequently, $\fp(M(BA''))$ is $\frac1{(8c^{2}_{M}c_\pi^2)}$-winning by Lemma~\ref{lemmLocalBilipWinningSetsweak} and its footnote.  An examination of the proof of Lemma~\ref{lemmLocalBilipWinningSetsweak} shows that we can make the same assertion for $\frac1{(8c^{2}_{M}c_\pi^2)}$-strong winning.  Now let $\a:= \frac1{(8c^{2}_{M}c_\pi^2)}$.  (Since $M$ is the identity for diagonal quadratic forms, this definition of $\a$ agrees with the one for those forms.)  Thus this shows the theorem for an arbitrary (rational, nondegenerate, indefinite) quadratic form.  
\end{rema}

\subsection{Proof of Theorem~\ref{thmBALightCone}}\label{secProofLightConeBA}

The key difference--indeed, simplification--between this case and the level-surface case is that one no longer needs to miss lines but only points because the role played by Proposition~\ref{propIntegerSepartionLW} is now played by Proposition~\ref{propIntegerSepartionLC}.  To prove the light-cone case, we follow the level-surface case and note the differences.  Define:  \[P_0 := \bigcup_{ \|U\|_{q_2} <1} X_{\|U\|_{q_2}} \bigg\backslash \{\langle 0, 0 \rangle\}.\]  Then $d_0$ is defined, analogously, with respect to $P_0$.  Force Player $B$'s balls to have radii strictly less than $R_0:=\frac 1 3 \min((2 \kappa_0)^{-1}, d_0/2, d_1/2, d_2/2,1/2, c_\pi /\sqrt{d}, R)$ where $d_1, d_2, c_\pi,$ and $R$ are the same as in the level-surface case.  Consequently, $B_1$ can contain at most one normalized (which, recall, means we divide by its sup norm) point $q_0$ of $P_0$.  Let $q = \pi(q_0)$.  

\subsubsection{Missing a point.}\label{secMissPoint} To miss $q$, we can, at random, pick a line $L$ through $q$ in $\mE$ and then follow Section~\ref{secMissLine} exactly or, note, that we can simplify that proof as follows.  Let $p :=c(B^\pi_1)$.  If $q$ does not lie on $\mN_p$, then the line through $p$ and $q$ and the line $\mN_p$ determine a $2$-plane $\mP_\mN$.  Now, one does not need Gram-Schmidt because there are only two directions $v$ and $-v$ in $T_p (\tilde{\partial} X) \cap \mP_\mN$.  Picking the direction that moves farther away from $q$ and following the rest of the proof in Section~\ref{secMissLine} shows that $A_1$ misses $q_0$.  If $q$ does lie on $\mN_p$, then any direction in $T_p (\tilde{\partial} X)$ will work, as they are all normal to $\mN_p$.  Thus, as in the level-surface case, we may, without loss of generality, assume that $B_1$ does not meet any normalized point of $P_0$.

\subsubsection{Essence of the proof}\label{secEssProofLC}Player $B$ has chosen $B_1$.  Define the constant \[c' := \min(\frac {(\a\b)^2}{4\kappa_0}, \frac {(\a\b)^2}{4}, \frac {\a\b\rho(B_1)}{4 \kappa_0}, \frac {\a\b\rho(B_1)}{4}).\]  Let $M_i$ be as in the level-surface case.  We window elements of $X_0(\ZZ)$ as follows:  \[P_1 := \bigcup_{1 \leq \|U\|_{q_2} \leq \frac{M_1}{\kappa_0}} X_{\|U\|_{q_2}}\] and, for natural numbers $i > 1$, \[P_i := \bigcup_{\frac{M_{i-1}}{\kappa_0} \leq \|U\|_{q_2} \leq \frac{M_i}{\kappa_0}} X_{\|U\|_{q_2}}.\]  
We see, as in the level-surface case, that the constant $c'$ is scaled correctly:

\begin{lemm}\label{lemmcprimecorrectLW}  
\begin{align*}
\textrm{We have} \quad c' \leq \frac 1 4 (\a\b)\rho(B_1) & \quad \textrm{when }  \quad i = 1, \\ 
\textrm{and} \quad \frac{c' \kappa_0}{M_{i-1}} \leq \frac 1 4 (\a\b)\rho(B_i) &  \quad \textrm{when }  \quad i > 1.
\end{align*}
\end{lemm}

The rest of the proof is analogous to the level-surface case, except that we replace Proposition~\ref{propIntegerSepartionLW}  by Proposition~\ref{propIntegerSepartionLC}, in which case we never need to miss lines, only points.\footnote{In each $BB_i$, there can be at most one normalized point of $P_i$ by Proposition~\ref{propIntegerSepartionLC}.  A small ball around this point is what must be missed.}  Therefore, we may replace references to Section~\ref{secMissLine} by Section~\ref{secMissPoint}.  (Of course, Lemma~\ref{lemmcprimecorrectLW} replaces Lemma~\ref{lemmcprimecorrect} and $1$ replaces $3\sqrt{|m|}$.)  Also, since we need only miss points, the $d=2$ and the $d \geq 3$ cases have the same proof, namely the proof that is analogous to the $d \geq 3$ level-surface case.

\subsection{The Hausdorff dimension of the set of badly approximable vectors}  For winning subsets of manifolds, one shows that they have full Hausdorff dimension using Lemma~\ref{lemmWinningFullHDManifolds}, a lemma that requires certain bilipschitz homeomorphisms:

\begin{lemm}\label{lemmVarietyHasBilpMaps}
Let $d \geq 1$.  For every point $p \in \partial X$, there exists an open neighborhood $U$ of $\partial X$ containing $p$ and a bilipschitz homeomorphism $\varphi:  U \rightarrow \varphi(U) \subset \RR^{d-1}$.
\end{lemm}

\begin{proof}
We may assume that $d \geq 2$, as the $d=1$ case is trivial.  Let us first consider the case where $p$ lies in exactly one face $\mF$ of $C^d$.  Then there exists a small enough open ball $U$ with center $p$ so that $U$ does not meet any $d-1$-dimensional face of $\mF$.  Consequently, we can consider $p$ and $U$ as lying in a hypersurface to which Lemma~\ref{lemmClosetoLinear} applies.\footnote{Replace $U$ with its intersection with the hypersurface in $\mF$.}  (The ball $U$ is small enough so that the Taylor approximation in the proof of Lemma~\ref{lemmClosetoLinear} applies.)  Let $f$ be as in the proof of Lemma~\ref{lemmClosetoLinear}, and thus if we write a point of $U$ as $p +Y$ (the vector $Y \in \RR^d$ is small), then $f(p+Y)=0$ is satisfied.  Now $p+Y$ has coordinates--translate the coordinate system to have origin at $p$--$(Y_1, \cdots, Y_{d-1}, -p_d\pm\sqrt{\a^{-1}_d[m-\a_1 (Y_1+ p_1)^2  - \cdots - \a_{d-1} (Y_{d-1}+p_{d-1})^2]})$.\footnote{This $m$ is different from the $m$'s in Theorems~\ref{thmBALevelSurface} and~\ref{thmBALightCone}.  Also, the ambiguity in signs is made unambiguous by $U$.}  Furthermore, we can write $p+Y$'s unique corresponding point on the tangent space at $p$ in coordinates as $q:=(Y_1, \cdots, Y_{d-1}, -\frac {1}{\a_d p_d }(\a_1 p_1Y_1 + \cdots + \a_{d-1} p_{d-1} Y_{d-1})),$ as in the proof of Lemma~\ref{lemmClosetoLinear}.

Define $\varphi$ as follows:  $\varphi(p + Y) = q$; evidently, it is a bijection.  Let $p+Y, p+Z \in U$.  Then $\|\varphi(p+Y) - \varphi(p+Z)\|^2_2 = (Y_1 - Z_1)^2 + \cdots + (Y_{d-1} - Z_{d-1})^2 + [\frac {1}{\a_d p_d }\big(\a_1 p_1(Y_1-Z_1) + \cdots + \a_{d-1} p_{d-1} (Y_{d-1}-Z_{d-1})\big)]^2$.  Now the last term is  just the square of \[\big|\frac {1}{\a_d p_d }(\a_1 p_1, \cdots, \a_{d-1}p_{d-1}) \cdot (Y_1- Z_1,  \cdots, Y_{d-1} - Z_{d-1})\big| \leq \|(Y_1- Z_1,  \cdots, Y_{d-1} - Z_{d-1})\|_2\] where the inequality is up to multiplication by a constant.  Consequently, we have \[\|(Y_1- Z_1,  \cdots, Y_{d-1} - Z_{d-1})\|_2\leq \|\varphi(p+Y) - \varphi(p+Z)\|_2 \leq c_1 \|(Y_1- Z_1,  \cdots, Y_{d-1} - Z_{d-1})\|_2\] for a constant $c_1\geq 1$ depending on $p$.  

For the inverse map, first recall that the $\sqrt{\cdot}$ function is Lipschitz on any proper, finite-length interval (in $\RR_{> 0}$) bounded away from $0$ because its derivative is bounded on that interval.  Now, for any $p +Y \in U$, the real-valued function $g(p+Y):=\a_d^{-1}[m-\a_1 (Y_1+ p_1)^2  - \cdots - \a_{d-1} (Y_{d-1}+p_{d-1})^2]$ is near zero if and only if $Y_d$ is near $-p_d$; recall from the proof of Lemma~\ref{lemmClosetoLinear} that $|p_d|$ is bigger than some positive constant.  Shrink $U$ if necessary to force all points in $U$ to have absolute value of the $d$-th coordinate (with respect to $p$ as origin) much less than this constant.  Hence, $g(U)$ is bounded away from zero.  And thus, in $U$, we have \begin{align*}\|p+Y -  & p -Z\|_2^2 \leq  (Y_1 - Z_1)^2 + \cdots +  (Y_{d-1} - Z_{d-1})^2 +  \\ & c_2' \frac 1 {\a^2_d} \bigg(\a_1 [(Y_1+ p_1)^2 - (Z_1+ p_1)^2] + \cdots + \a_{d-1} [(Y_{d-1}-p_{d-1})^2-(Z_{d-1}-p_{d-1})^2]\bigg)^2\end{align*} for some positive constant $c_2'$ (the bilipschitz constant of $\sqrt{\cdot}$).  Also, since $Y$ and $Z$ are small, we have that \[|(Y_i+ p_i)^2 - (Z_i+ p_i)^2| = |Y_i - Z_i| |Y_i+Z_i +2 p_i| \leq |Y_i- Z_i|\]  where the inequality is up to multiplication by some positive constant.
Now $(x+y)^2 \leq 3 (x^2+y^2)$.  Therefore, there exists a constant $c_2 \geq 1$ depending on $p$ such that  \[\|(Y_1- Z_1,  \cdots, Y_{d-1} - Z_{d-1})\|_2\leq \|p+Y - p+Z\|_2 \leq c_2 \|(Y_1- Z_1,  \cdots, Y_{d-1} - Z_{d-1})\|_2.\]  We have shown that $\varphi$ is bilipschitz for $p$ lying in only one face.

When $p$ lies in more than one face, we can follow the technique in Section~\ref{subsubsecHaCoC} and repeat the proof for the previous case with $\pi(U)$ replacing $U$.  Then the desired map $\varphi \circ \pi$ is bilipschitz. \end{proof}

\begin{rema}
To show the lemma for an arbitrary (rational, nondegenerate, indefinite) quadratic form, one may use $\varphi\circ \fp \circ M^{-1}$ if the point $\fp \circ M^{-1}(p)$ lies on only one face of $C^d$ or $\varphi \circ \pi \circ \fp \circ M^{-1}$ if it lies on more than one face--see Remark~\ref{rmkGeneralQuadForm}.
\end{rema}

\noindent These bilipschitz homeomorphisms are charts, which, together, constitute an atlas for $\partial X$.


\section{Auxiliary observations on badly approximable vectors}\label{secAuxObsonBA}  In this section, we prove the auxiliary results stated in the Introduction.  Let $Q$ be a diagonal, rational, nondegenerate, indefinite quadratic form;\footnote{Letting $Q$ be diagonal does not result in any loss of generality; for an arbitrary (rational, nondegenerate, indefinite) quadratic form, one can follow the proof in Remark~\ref{rmkGeneralQuadForm} and perform the analogous changes to the proofs in this section.} denote the last variable in the form $x_{d+1}$ by $y$.  Since $Q$ is nondegenerate, the coefficient of $y$ is nonzero.  We may divide by the negation of this coefficient without loss of generality.  Thus, we may assume that $Q = q - y^2$ for some diagonal, nondegenerate, rational quadratic form $q$--note that, since $Q$ is indefinite, $q$ cannot be negative definite.  We use the notation $\langle v,w \rangle_1$ to denote a $d+1$-vector comprised of a $d$-vector $v$ and a real number $w$.  Define the set \[V_q:= \{x \in \RR^{d} \mid q(x) = 1\}.\]   For $v \in V_q$, it is immediate that $Q(\langle v,y \rangle_1) = 0$ is equivalent to the condition that $y = \pm 1$.  
Unlike the natural splitting of Section~\ref{secProofofConjecture}, the splitting $\langle \cdot,\cdot \rangle_1$ is artificial, but it leads to the following useful lemma:

\begin{lemm} \label{lemmProjBA}
Let $m \in \QQ$ and $\psi(t) = t^{-s}$ for $s > 0$.  Let $v \in \RR^d$.  If \begin{enumerate}
\item $v/\|v\| \in V_q$ and
\item there exists a constant $c(v)>0$ such that, for all $\langle x,y \rangle_1 \in X_m(\ZZ) \big\backslash \ZZ^d \times \{0\}$, we have \[ \|\frac v {\|v\|}  - \frac x y\| \geq c\psi(|y|),\] then $\langle \frac{v}{\|v\|}, 1 \rangle_1 \in BA_{\partial X}^\psi(X_m(\ZZ))$.\footnote{Since $\|\langle \frac v{\|v\|},1 \rangle_1 - \frac{\langle x,0 \rangle_1}{\|\langle x,0 \rangle_1\|}\| \geq 1$, we need not consider approximation by any  $\langle x,0 \rangle_1 \in X_m(\ZZ) \backslash \{\mathbf{0}\}$.}
\end{enumerate}

\end{lemm}

\begin{proof}
By condition (1), $Q(\langle \frac{v}{\|v\|}, 1 \rangle_1) =0$.  Hence, $\langle \frac{v}{\|v\|}, 1 \rangle_1 \in \partial X$.

By condition (2), we obtain $\|\frac {\langle v, \|v\|\rangle_1} {\|v\|}  - \frac {\langle x, y \rangle_1} y\| \geq c\psi(|y|)$.  Depending only on $\langle \frac{v}{\|v\|}, 1 \rangle_1$, there is a positive constant $c'$ (which gives a bound for the local distance distortion when changing from one radial projection to the other\footnote{The two radial projections are onto $C^d$ and onto the $d$-dimensional hyperplane whose last coordinate is~$1$.}) such that $ c' \|\frac {\langle v, \|v\|\rangle_1} {\|\langle v, \|v\|\rangle_1\|}  - \frac {\langle x, y \rangle_1} {\|\langle x, y \rangle_1\|}\| \geq \|\frac {\langle v, \|v\|\rangle_1} {\|v\|}  - \frac {\langle x, y \rangle_1} y\|$.  Since $\|\langle x, y \rangle_1\| \geq |y|$, the desired result is immediate.

\end{proof}

\begin{rema}
If $\|v\| \leq 1$, then replacing $\|v\|$ with $1$ everywhere in the same proof above shows that $\langle v, 1 \rangle_1\in BA_{\partial X}^\psi(X_m(\ZZ))$.  
We can also replace $X_m(\ZZ)$ by any subset of $\ZZ^{d+1}$.
\end{rema}

This lemma helps to prove our results for $d=1$:

\begin{proof}[Proof of Theorem~\ref{thmDrCor}]

Note that $\{(\pm 1/\sqrt{\a}, \pm 1)\} = V_q \times \{\pm 1\}$.  Let $v \in V_q$.  

Case:  $\sqrt{\a}$ is rational.  Write $\sqrt{\a} = p/p'$ where $p,p'$ are relatively prime positive integers.  First, assume that $m \neq 0$.  Let $\varepsilon := \min\{|z| \mid z \in \frac 1 p \ZZ \backslash \{0\}\}$.   If there exists some $(x, y) \in \ZZ^2$ such that $\|yv - x\| < \varepsilon$, then $x = yv$ and $Q(x,y) = 0$.  Consequently, $(x,y) \notin X_m(\ZZ)$.  Since $|v| \leq 1$, it follows from Lemma~\ref{lemmProjBA} that $\langle v, 1\rangle_1 \in BA_{\partial X}^\psi(X_m(\ZZ))$.  Since the negation of a badly approximable vector is badly approximable, $\{(\pm 1/\sqrt{\a}, \pm 1)\} \subset BA_{\partial X}^\psi(X_m(\ZZ))$.  Now assume $m=0$.  Consequently, $(p',p)$ or $(-p',p)$ are in $X_0(\ZZ)\backslash \{\boldsymbol{0}\}$.  This implies that $(v,1) = \frac{(p',p)}{\|(p',p)\|}$ or $(v,1) = \frac{(-p',p)}{\|(-p',p)\|}$.  Negating these last two shows that each of the four points of $\partial X$ is approximable by itself; thus, $BA_{\partial X}^\psi(X_0(\ZZ)) = \emptyset$. 

Case:  $\sqrt{\a}$ is irrational.  Thus, $v$ is an irreducible quadratic and badly approximable as a real number.  Therefore, there exits a constant $ 1/2 >k(v) >0 $ such that, for all $y \in \ZZ \backslash \{0\}$ and all $x \in \ZZ$, we have $\|yv - x\| \geq k/|y|$.  Let $\varepsilon >0$ be chosen so that $\a \varepsilon^2 <  k \sqrt{\a}$.  Now if there exists some $(x, y) \in \ZZ^2 \backslash \{(0,0)\}$ such that $\|yv - x\| < \varepsilon$, then $x = yv + \ell/|y|$ where $|\ell| \geq |k|$ and $\ell^2/y^2 < \varepsilon^2$.
Note that $Q(\langle x, y\rangle_1) = 2 \ell \sqrt{\a} + \ell^2 \a/y^2$ or $-2 \ell \sqrt{\a} + \ell^2 \a/y^2$.  Therefore, $|Q(\langle x, y\rangle_1)| \geq 2 |\ell | \sqrt{\a} - \ell^2 \a/y^2 \geq  k \sqrt{\a}$.  Consequently, for all $|m| < k \sqrt{\a}$, $(x,y) \notin X_m(\ZZ)$.\footnote{For both $v$ and $-v$, the same $k$ can be used.}


\end{proof}

Using the lemma again, we can show that, for $m \neq 0$, $\psi(t) = t^{-s}$, and $s>2$, all vectors are badly approximable:

\begin{proof}[Proof of Theorem~\ref{thmApproximationNotPossSbigger2}]
We show that $\partial X \backslash BA_{\partial X}^\psi(X_m(\ZZ)) = \emptyset$.  Assume not.  Let $v' \in \partial X \backslash BA_{\partial X}^\psi(X_m(\ZZ))$.  By permuting the coordinates if necessary and noting that $-v'$ must also not be badly approximable, we may assume that $v' := \langle v, 1\rangle_1$.  Thus $v \in V_q$.  Then by Lemma~\ref{lemmProjBA}, we have that infinitely many $\langle x,y \rangle_1 \in X_m(\ZZ)\big\backslash \ZZ^d \times \{0\}$ satisfy the inequality \[ \| y v   - x \| < |y|^{1-s}.\]  Consequently, $x = yv + (c_1 |y|^{1-s}, \cdots, c_d |y|^{1-s})$ where all of the $c_i$s are smaller than $1$ in absolute value.

Since $\langle x, y \rangle_1 \in X_m(\ZZ)$, we have that $m = q(x) - y^2$.  But, $q(x) -y^2 = 2 k_1c_1 v_1 y |y|^{1-s} + k_1c^2_1 |y|^{2-2s} + \cdots+ 2 k_d c_dv_d y |y|^{1-s} + k_dc^2_d |y|^{2-2s}$ where $k_i$s are the coefficients of $q$ and $v_i$s are the components of $v$ (note $\|v\|\leq 1$).  Thus, $|q(x) - y^2| \leq \a |y|^{2-s} + \b |y|^{2 - 2s}$ where $\a$ and $\b$ are bounded constants.  Since there are infinitely many $\langle x, y \rangle_1 \in X_m(\ZZ)$ that satisfy that last inequality (equivalent to $|y| \rightarrow \infty$), the constant $m =0$, a contradiction.
\end{proof}

\begin{rema}\label{remaKleinbocksRemark}
The proof of Theorem~\ref{thmApproximationNotPossSbigger2} is even stronger than the assertion of the theorem:  for $s >2$ (and $m \neq 0$), one cannot approximate even with a sequence of vectors on $X_m$ with real components.  Thus, unlike for $s=2$--see Example 5.2 of~\cite{GS}, for $s >2$ (and $m \neq 0$), the notion of approximation is not meaningful.\footnote{Thanks to D.~Kleinbock for noticing that the proof is stronger than the assertion being proved.}  
\end{rema}

Finally, we show that $X_m(\ZZ)$, and not $\ZZ^{d+1}$, is the appropriate subset of the integer lattice to use for the set of badly approximable vectors.  Let $\|x\|_{\ZZ} = \inf_{z \in \ZZ^{d+1}} \|x-z\|$.  To show this, we first need a general lemma:

\begin{lemm}\label{lemmBAInfNormConvert}
Let $v \in \partial X$.  Then ${v} \in BA_{\partial X}^\psi(\ZZ^{d+1})$ if and only if there exists a constant $1 \geq c>0$ (depending on ${v}$) such that, for all $x \in \ZZ^{d+1}\backslash \{\boldsymbol{0}\}$, we have $\|\|x\|{v}\|_{\ZZ} \geq c$.\footnote{This lemma is still true if we replace $\ZZ^{d+1}$ with a subset everywhere.}
\end{lemm}

\begin{proof}
Let ${v} \in BA_{\partial X}^\psi(\ZZ^{d+1})$.  Then there exists a constant $1 \geq c>0$ such that, for all $x \in \ZZ^{d+1}\backslash \{\boldsymbol{0}\}$, we have $\big{\|}x - \|x\| {v}\big{\|} \geq c$.  Let $y \in \ZZ^{d+1}$ such that $\|y\| = \|x\|$.  Then $\big{\|}y - \|x\| {v}\big{\|} \geq c$.  

Now let $y \in \ZZ^{d+1}$ be such that $\|y\| \neq \|x\|$.  Assume that $\big{\|}y - \|x\| {v}\big{\|} < c$.  Thus, $\big{|}\|y\| - \|x\| \|{v}\|\big{|} < c$.  Since $\|{v}\| =1$, we obtain $\|y\| = \|x\|$, indicating that our assumption, $\big{\|}y - \|x\|v\big{\|} < c$,  is false.  
Consequently, we may conclude that ${v} \in BA_{\partial X}^\psi(\ZZ^{d+1})$ is equivalent to the existence of a constant $1 \geq c>0$ (depending on ${v}$) such that, for all $x \in \ZZ^{d+1}\backslash \{\boldsymbol{0}\}$, we have $\|\|x\|{v}\|_{\ZZ} \geq c$.
\end{proof}

\begin{proof}[Proof of Theorem~\ref{theoFullLatticeTooBig}]
Without loss of generality, we may assume that $s=1$.  Let $v \in \partial X$ and $T_v$ denote translation by $v$ on the torus $\TT^{d+1}:=\RR^{d+1}/\ZZ^{d+1}$.  It is well-known that $\overline{\Or_{T_v}(\boldsymbol{0})}$ is a finite union of $k$-dimensional tori, where $0 \leq k \leq d$, and that the restriction of $T_v$ to each of these tori is minimal.  Now $\boldsymbol{0}$ lies on one of these tori.  Fix this tori.  If $k >0$, then, because the orbit closure is a finite union of closed subsets of  $\TT^{d+1}$, this tori must contain $T^j_v(\boldsymbol{0})$ for some $j \in \ZZ \backslash \{0\}$.  By minimality, the orbit of $jv$ under the restriction is dense, and hence, for every $\varepsilon >0$, there exists an $n \in \ZZ \backslash \{0\}$ such that $\|nj{v}\|_{\ZZ} < \varepsilon$.  Thus, $v \notin BA_{\partial X}^\psi(\ZZ^{d+1})$ by Lemma~\ref{lemmBAInfNormConvert}.

If $k=0$, then the orbit closure is a finite union of points and thus the translation is periodic and the result is obvious.

 \end{proof}

\begin{rema}
Both Lemma~\ref{lemmBAInfNormConvert} and Theorem~\ref{theoFullLatticeTooBig} can be asserted with $C^d$ replacing $\partial X$ everywhere--we do not use the variety in these proofs at all.
\end{rema}

\section{Conclusion}

In this paper, we have shown that the set of vectors that are badly approximable by the appropriate (and natural) set of integer lattice points is $\a$-strong winning (and $\a$-winning) and whence has full Hausdorff dimension and the countable intersection property.  For diagonal (rational, nondegenerate, indefinite) quadratic forms, the constant $\a$ (which is called the winning parameter~\cite{ET}) depends only on the dimension of the light-cone from which both our badly approximable vectors and the integer lattice points that we use to approximate come.  As mentioned (in the beginning of Section~\ref{secProofLevelSurfaceBA}), one should, for these diagonal forms, be able to replace $\a$ by the largest possible winning parameter of $1/2$ using a certain technique, which is found in Schmidt's original paper~\cite{Sch2} and involves replacing each iteration of the game that we played with a finite number of iterations of a certain kind--roughly, one can characterize these additional steps as ``pushing'' in a certain direction (namely, the chosen direction $v$ from our proof, which, recall, moves normally away from the neighborhood of the line or point to be missed).  

For arbitrary (rational, nondegenerate, indefinite) quadratic forms, the winning parameter $\a$ depends on the dimension of the light-cone and on the form as shown in Remark~\ref{rmkGeneralQuadForm}.  The proof in the remark concludes with an application of the bilipschitz map $\fp \circ M$ after the induction is finished (see the remark for the definition of the notation).  This, however, is incompatible with the technique of Schmidt mentioned in the previous paragraph because the winning parameter could decrease as stated in Lemma~\ref{lemmLocalBilipWinningSetsweak} and its footnote.  To surmount this obstruction, we outline another proof of the assertion in the remark.  Since the map $\fp \circ M$ is bilipschitz, we can treat it in a manner similar to the bilipschitz map $\pi$ (Section~\ref{subsubsecHaCoC}), namely apply it at every stage of the induction proof as we did $\pi$--note that this application of $\fp \circ M$ is done after $A_i$ is chosen in the $i$-th step of the induction and hence all the geometry is handled as in the diagonal form case.  This gives an alternate proof (alluded to in the remark) that for arbitrary forms the desired set is $\a$-winning and $\a$-strong winning (where $\a$ is as in the remark).  Now combining this alternate proof with the technique of Schmidt, we should be able to replace the winning parameter $\a$ with the largest possible winning parameter $1/2$ for arbitrary forms.  Yet, finding the optimal winning parameter is of minor interest because doing so does not produce, to the author's knowledge, any interesting new corollaries.

To obtain interesting new corollaries, we can restrict the game to a large enough fractal lying on the variety and on the cube (i.e. $\partial X$), as, for example, \cite{Fi},~\cite{Fi2},~\cite{BFK},~\cite{BBFK},~\cite{BHKV},~\cite{KTV},~\cite{KW2}, and~\cite{ET} do for Euclidean space.  These fractals arise from (the pullback under some suitable coordinate charts of) the support of what-are-called absolutely friendly and fitting measures (see Section~2.3 of~\cite{ET} for the definitions).\footnote{The development of absolutely friendly measures can be traced in, for example, \cite{KLW} and~\cite{PV}.  Fitting measures come from~\cite{Fi2}.}   The absolutely friendly property (in particular, what-is-called the absolutely decaying part of this property) allows us, by definition, to find other points of the fractal away from a small-enough thickened (codimension-one with respect to $\partial X$) hyperplane.  For the technique in this paper to apply, we need to require $d \geq 4$, so that the variety has dimension at least $3$.  Now recall in our proof that there is at most one line or point to miss in any $BB_i$.  This line and the normal line $\mN_p$ (lying in $\mE$) of the center $p$ of $BB_i $ form a $3$-space $\mS$, which by Gram-Schmidt is spanned by orthonormal vectors along $\mN_p$ and in $T_p(\tilde{\partial} X)$.  Gram-Schmidt further gives us a vector $v \in T_p(\tilde{\partial} X)$ normal to $\mS$.  Now thicken $\mS$ in the $v$ and $-v$ directions.  We may move away from the thickening in either the $v$ or $-v$ directions.  Combining this ability to find a center for Player $A$'s ball that lies on the fractal away from the thickening with the technique of this paper, we can replace $\partial X$ with (the pullback under some suitable coordinate charts of) the support of an absolutely decaying measure (which, of course, lies in $\partial X$) in Theorems~\ref{thmBALightCone} and~\ref{thmBALevelSurface} and still retain the conclusion (except that $\a$ may be smaller and may depend on the fractal).\footnote{The reason for our use of Lemma~\ref{lemmClosetoLinear} is to allow us to regard (up to some very small error that we called $\varepsilon$) the variety as Euclidean space locally.  Since the game is played locally as mentioned above, one would expect that these two techniques combine together.  We should note that we thicken enough to, not only miss the set $L'$ from Section~\ref{secEssenseProofLW}, but also miss three times the radius of Player $A$'s ball.  Moreover, we no longer consider $B_i$, but the thickening is considered with respect to the ball with the same center as $B_i$, albeit with radius $(1 - 3 \a c_\pi^2) \rho(B_i)$--picking a center under such restrictions guarantees that $AAA_i$ lies in $B_i$.}    In Corollaries~\ref{corBALightCone} and~\ref{corBALevelSurface}, we need to further restrict to absolutely friendly and fitting measures (see Section~1.1 of~\cite{ET} for an indication of why) to conclude that the badly approximable vectors in the fractal have the same Hausdorff dimension as the fractal.\footnote{Note that the integer lattice points that we use to approximate with are the same regardless of what fractal we restrict our badly approximable vectors to.}  But, even this more restrictive class of such measures gives rise to a familiar litany of fractals:  the Cantor set, the Koch curve, the Sierpinski gasket, and so on  (see Corollary 5.3 of~\cite{Fi2} and Theorem 2.3 of~\cite{KLW} for more details).  The restriction to these fractals significantly extends our Diophantine result, but we should note that the winning parameter may be strictly less than $1/2$, as shown in Section 5 of~\cite{Fi} for some of these fractals lying in Euclidean space.

Another way to generalize our results is to consider another variant of Schmidt games.  Let us consider absolute-winning Schmidt games, for which our proof technique only allows generalization in the light-cone case and in the $d=2$ level-surface case.\footnote{Absolute winning is not meaningful for $d=1$.}  This variant of the game is introduced, along with strong-winning Schmidt games, in Section 1 of~\cite{Mc}.  More details, including the definition, can be found in that paper.  To see why we can generalize in the aforementioned cases, note that, for absolute winning, we take $\a = \b$ (and $0<\b<1/3$ is arbitrary) and, since we are playing for Player $A$, agree to take $\rho(A_i) = \b \rho(B_i)$.  Now, with $\a = \b$, Lemmas~\ref{lemmcprimecorrectLW} and~\ref{lemmcprimecorrect} imply that $A_i$ is big enough to cover the requisite-sized ball $B$ around the one normalized point of $P_i$--for absolute winning, the goal of Player $A$ is not to miss this point, but to contain this point and this little ball $B$ around this point.  Now if $B$ does not meet $\partial X$, Player $A$ can freely pick (according to the rules of the absolute game) because $\cap B_i$ lies on $\partial X$ and thus misses $B$.  If $B$ does meet $\partial X$, then, since $A_i$ is much larger than $B$, pick $A_i$ to contain $B \cap \partial X$.  Hence, we can conclude that our set of badly approximable vectors is absolutely winning.

\subsection{Badly approximable vectors and dynamical systems}  In Euclidean space, the notion of badly approximable vectors (or linear forms) relates, in one way, to the theory of dynamical systems via toral translations (a well-known relation, stated in the Introduction of~\cite{BHKV}, for example).  In a similar way, our notion of badly approximable vectors relates to toral translations of higher rank.  But, with us, the higher-rank action is no longer given by the whole lattice $\ZZ^{d+1}$ because, as stated in Theorem~\ref{theoFullLatticeTooBig}, there are no such badly approximable vectors; our action is, instead, given by the intersection of this lattice with a hypersurface, namely a natural, geometric restriction.  Thus, under our natural action, we see that there exists a correspondence between being badly approximable and having a nondense orbit and, what is more, the set of such has zero measure, but has full Hausdorff dimension and is strong winning, as in the Euclidean case (Theorem~1.1 and Section~5.1 of~\cite{ET}).  Other (related) papers on badly approximable and winning include~\cite{Sch2}, \cite{Sch3}, \cite{Da}, \cite{Fi}, \cite{Fi2}, \cite{KW}, \cite{KW2}, \cite{Mc}, and~\cite{T5}.  On the (purely) dynamical side, papers on nondense orbits and winning include~\cite{Da2}, \cite{BBFK}, \cite{BFK}, \cite{Fa}, \cite{KW3}, and~\cite{T4}.

\subsection{Questions}  Our results have raised a few questions.  First, the question of whether the $d \geq 3$ level-surface case of our results generalizes to absolute winning is open.  Since our proof requires lines in this case, not just points, it does not immediately generalize; however, the lines requirement may be an artifact of the proof.  Second, the question of whether the $d=2$ and $d=3$ cases of our results generalize for the fractals mentioned above is open--note that this question is not meaningful for $d=1$, as the variety on the cube is just a finite set of points.  Third, what happens for $d=1$ in the level-surface case beyond what is shown in Theorem~\ref{thmDrCor}?  
Fourth, are Theorems~\ref{thmBALevelSurface} and~\ref{thmBALightCone}  still true if one uses a different approximating function $\psi(t)$?  And finally, fifth, how do absolute winning and fractals relate?  To play absolute-winning games on fractals, one may need to find a more restrictive class of fractals because the ability to find points of the fractal away from thickened codimension-one hyperplanes only exists for small thickenings, but, to play an absolute game, one needs to thicken enough to cover $A_i$, which can be quite big (as much as $\rho(B_i)/3$).\footnote{One must cover $A_i$ in order to find a center for $B_{i+1}$.}

\section{Appendix}  

In this section, we collect some technical lemmas.  

\subsection{Lemmas on winning sets}  
An $n$-dimensional manifold $M$ (with or without boundary) is metrizable.\footnote{We do not assume that $M$ is second countable.}  Pick a metric on $M$ and impose the restriction that $M$ be a complete metric space.  Hence, one can play a Schmidt game on $M$.  Also impose the restriction that, for every point $p$ of $M$, there exists an open neighborhood $U$ containing $p$ and a bilipschitz homeomorphism $\varphi := \varphi_p: U \rightarrow \varphi(U) \subset \RR^n$ if $p$ is an interior point or $\varphi := \varphi_p: U \rightarrow \varphi(U) \subset \HH^n$ if $p$ is a boundary point (the metric on $\RR^n$ or $\HH^n$ is given by $\|\cdot\|$).\footnote{To obtain $\dim(S)=n$ for an $\a$-winning set $S$ of $M$, we only need one bilipschitz homeomorphism from some open set of $M$--it is not necessary to obtain one for every point of $M$.}   Such a collection of bilipschitz homeomorphisms could form an atlas, but this is not necessary.  Let $(M, \{\varphi_p\})$ denote such a manifold with such maps.  Let $0 < \a <1$, and let $\dim(\cdot)$ denote the Hausdorff dimension.  In Lemma~\ref{lemmWinningFullHDManifolds}, we show that, like winning subsets of Euclidean space, an $\a$-winning subset of such a manifold has full Hausdorff dimension; this lemma requires the observation that a bilipschitz homeomorphism preserves winning sets:

\begin{lemm}\label{lemmLocalBilipWinningSetsweak}
Let $\varphi: U \rightarrow \varphi(U) \subset \RR^n$ be a bilipschitz homeomorphism with constant $c_\varphi \geq 1$.  Let $U' \subset U$ be a closed ball such that $\varphi(U')$ is bounded in $\RR^n$, and let $S \subset U'$ be an $\a$-winning set of $U'$.  Then $\varphi(S)$ is a $(c_\varphi^{-2}\a)$-winning set of $\varphi(U')$.\footnote{The lemma also holds in general--same proof--for a bilipschitz homeomorphism $\varphi: X \rightarrow \tilde{X}$ and an $\a$-winning set $S \subset X$ where $X, \tilde{X}$ are complete metric spaces.} 
\end{lemm}

\begin{proof}
The constant $0<\a<1$ is fixed, but the constant $0<\b<1$ is arbitrary.  We play a $(c_\varphi^{-2} \a, \b)$-game on $\varphi(U')$.  Player $B$ picks $\tilde{B}_1\subset \varphi(U')$.  Then there exists a unique closed ball $B_1:=B(\varphi^{-1}(c(\tilde{B}_1)), c_\varphi^{-1} \rho(\tilde{B}_1)) \subset \varphi^{-1}(\tilde{B}_1).$  Player $A$ uses the winning strategy for the $(\a, c_\varphi^{-2}\b)$-game on $U'$ to choose $A_1 \subset B_1$.  

Now there exists a unique closed ball $\tilde{A_1} := B(\varphi(c(A_1)), c_\varphi^{-1} \rho(A_1)) \subset \varphi(A_1) \subset \tilde{B}_1$.  Note that $\rho(\tilde{A}_1) = c_\varphi^{-2} \a \rho(\tilde{B}_1)$.  The ball $\tilde{A_1}$ is Player $A$'s choice for the game on $\varphi(U')$.  Player $B$ now picks $\tilde{B_2} \subset \tilde{A}_1$ (note that the radius of $B_2$ is determined in two ways in this proof--and they agree) and, by induction, one repeats the above to show that $\cap_i \tilde{B}_i = \varphi(\cap_i B_i) \in \varphi(S),$ thereby implying the desired result.
\end{proof}

\begin{lemm}\label{lemmWinningFullHDManifolds}
Let $S$ be an $\a$-winning subset of $(M, \{\varphi_p\})$.  Then $\dim(S)$ is equal to $n$, the dimension of $M$ (more precisely and stronger, given any open neighborhood $V$ of $M$, $\dim(S \cap V) =n$.)
\end{lemm}

\begin{proof}
All balls of $\RR^n$ in this proof are $\|\cdot\|$-balls.  
Pick a point $p \in M$.  Let us first consider $p$ to not be on the boundary of $M$.  Let $U$ be an open neighborhood of $p$ for which we have a bilipschitz homeomorphism $\varphi: U \rightarrow \varphi(U) \subset \RR^n$--we may shrink $U$ so that $U$ does not meet the boundary of $M$.  Now we may assume that $\varphi(U)$ is bounded in $\RR^n$ or, otherwise, replace $U$ with an open ball around $p$ contained in the preimage of some open ball around $\varphi(p)$ contained in $\varphi(U)$.  Let $U' \subset U$ be a slightly smaller, closed ball and $U'' \subset U'$ an open ball with the same center as $U'$, albeit with $1/2$ the radius.  Let $C \subset \varphi(U'')$ be a closed ball of $\RR^n$.  

Let $C' := \varphi^{-1}(C)$.  Since $\varphi(U')$ is a closed, bounded subset of $\RR^n$, it is compact and so is its preimage $U'$; thus $U'$ is a complete metric space on which we can play the game.  Then we claim that the set $(C' \cap S) \cup U' \backslash C'$ is $\a$-winning.  Player $B$ picks $B_1 \subset U'$.  Now Player $A$ uses the winning strategy (for the game on $M$) to obtain that $\cap B_i \in S$.  Also, $\cap B_i \in U'$.  Therefore, $S \cap U'$ is $\a$-winning; whence $(C' \cap S) \cup U' \backslash C' \supset S \cap U'$ is $\a$-winning for the game played on $U'$.
Let $c_\varphi \geq 1$ be the bilipschitz constant for $\varphi$.  Now Lemma~\ref{lemmLocalBilipWinningSetsweak} implies that $\varphi((C' \cap S) \cup U' \backslash C')$ is $(c^{-2}_\varphi\a)$-winning for the game played on $\varphi(U')$.

Moreover, we claim that the set $(C \cap \varphi(S)) \cup \RR^n \backslash C$ is $(c^{-2}_\varphi\a)$-winning for a game played on $\RR^n$.  Note that the closest that any point of $C$ is to $\varphi(\partial U')$ is $\frac 1 {2 \rho(U') c_\varphi}$; freely play the game until the diameter of the balls are smaller than this distance.  Without loss of generality, we may assume that $B_1$ satisfies this diameter restriction.  If $B_1 \subset \varphi(U')$, then Player $A$ uses the winning strategy from the game on $\varphi(U')$; otherwise, Player $A$ plays freely.  This shows that $(C \cap \varphi(S)) \cup \RR^n \backslash C$ is a $(c^{-2}_\varphi\a)$-winning subset of $\RR^n$.

Now, we can tile $\RR^n$ using translations $\psi_j$ of $C$; these are isometries, and thus $\cap_j \psi_j \big(C \cap \varphi(S)) \cup \RR^n \backslash C\big)$ is $(c^{-2}_\varphi\a)$-winning.  Since winning subsets of Euclidean space have full Hausdorff dimension, we have that $\dim(C \cap \varphi(S)) =n$.  Since Hausdorff dimension is also preserved under bilipschitz maps, we have that $\dim(C' \cap S) = n.$  This already implies that $\dim(S) =n$.

Finally, let us consider $p$ to be on the boundary of $M$.  Let $U$ be an open neighborhood of $p$ for which we have a bilipschitz homeomorphism $\varphi: U \rightarrow \varphi(U) \subset \HH^n$.  Repeat the above with $\HH^n$ replacing $\RR^n$ (and taking care that the balls $C, U', U''$ are either centered around $p$ or $\varphi(p)$) up to the point where one has shown that $(C \cap \varphi(S)) \cup \HH^n \backslash C$ is a $(c^{-2}_\varphi\a)$-winning subset of $\HH^n$.  Let $-\HH^n$ denote the reflection of $\HH^n$ across its boundary hyperplane.  Now we play a game on $-\HH^n \cup (C \cap \varphi(S)) \cup \HH^n \backslash C$.  Player $B$ picks $B_1$.  If $c(B_1) \in \HH^n$, then Player $A$ picks $A_1$ so that it completely lies in $\HH^n$.  Then $B_2$ is picked and Player $A$ uses the winning strategy for the game on $\HH^n$.  Otherwise, $c(B_1)$ lies in $-\HH^n$ away from the boundary hyperplane.  Consequently, Player $A$ picks $A_1$ to lie completely in $-\HH^n$ and henceforth plays freely; whence, $-\HH^n \cup (C \cap \varphi(S)) \cup \HH^n \backslash C$ is a $(c^{-2}_\varphi\a)$-winning subset of $\RR^n$.  As before, we can tile $\RR^n$ using $C$, which implies $\dim(C' \cap S)=n$.

\end{proof}

\subsection{Geometric lemmas}  Let $\|\cdot\|_q$ be a norm on $\RR^k$ given by a positive-definite quadratic form $q=a_1 Y^2_1 +\cdots + a_k Y^2_k$, and let $c_{2q}>0$ be the constant such that \[\frac 1 {c_{2q}} \|\cdot\|_q \leq \|\cdot\|_2 \leq c_{2q} \|\cdot\|_q.\]  For any real number $r >0$, define the set $S_r := \{v \in \RR^k \mid \|v\|_q = r\}$; note that this set is an ellipsoid--a bounded $k-1$-dimensional smooth manifold, which is a closed subset of $\RR^k$ and thus compact.

\begin{lemm}\label{lemmEllipsoidMissesLine}
Let $r>0$ be a real number.  Then there exists a real number $c_q>0$, depending only on $q$, such that, for all $w \not\sim v \in S_r$ and all real  numbers $\g \geq 0$, we have \[\|w - \g v\|_q > c_q \|w - v \|_q.\]
\end{lemm}

\begin{proof}
Assume not.  Then, for all $c_q >0$, there exist $w \not\sim v \in S_r$ and there exists $\g \geq 0$ such that \[\|w - \g v\|_q \leq c_q \|w - v \|_q.\]  Now $c_q \|w - v \|_q \leq 2c_qr$ and $\|w - \g v\|_q \geq \big|\|w\|_q - \g\|v\|_q\big|= r\big|1-\g\big|.$  Hence, \[\big|1-\g\big| \leq 2c_q.\]

Now $S_r$ and the interval $[-2,2]$ are compact metric spaces, and thus $S:= S_r \times S_r \times [-2, 2] \times [-2,2]$ is compact and metrizable.  Letting $c_q = n^{-1}$ for all $n \in \NN$ creates a sequence in $S$ satisfying  \begin{eqnarray}\label{eqnNormConvert2}\frac{\|w - \g v\|_q} {\|w - v \|_q}\leq c_q.\end{eqnarray}  For the limit of any subsequence, $c_q =0$ and $\g = 1$.  Pick a convergent subsequence.  Let $(w_0, v_0)$ be the limit point in $S_r \times S_r$.  If $w_0 \neq v_0$, then we have $1 \leq 0$, a contradiction.

If $w_0 = v_0$, then, for elements of the subsequence $(w_n, v_n)$ with large enough index $n$, $\|w_n - v_n \|_q$ is small, and thus $w_n$ lies very nearly on the tangent space of $S_r$ at $v_n$; denote this tangent space by $T_{v_n}$.  Let $L_n$ denote the line segment from the terminal point of $w_n$ to the terminal point of $v_n$.  The line segment $L_n$ almost lies in $T_{v_n}$.  Let $N_{v_n}$ denote the outward unit normal vector at $v_n$ to $S_r$ in $\RR^k$.

Let $\cdot$ denote the usual inner product:  given vectors $V:=(V_1, \cdots, V_k)$ and $W:=(W_1, \cdots, W_k)$, then $V \cdot W := V_1W_1 + \cdots +V_kW_k$.  Now the angle (with respect to the inner product $\cdot$) of a vector $u:=(u_1, \cdots, u_k) \in S_r$ with $N_u$ is bounded uniformly (i.e. the bound depends only on $q$) away from being orthogonal.  More precisely, \[\frac{u}{\|u\|_2} \cdot \frac{N_{u}}{{\|N_u\|_2}}  =  \frac {\frac 1 {r^2} (u_1, \cdots, u_k) \cdot (2 a_1 u_1, \cdots, 2 a_k u_k)}{\frac 1 {r^2}\|u\|_2 \ \|(2 a_1 u_1, \cdots, 2 a_k u_k)\|_2} \geq \frac 1 {c_{2 q} \sqrt{a}}\] where $a = \max(a_1, \cdots, a_k)$.  (Note that $\|(a_1 u_1, \cdots,  a_k u_k)\|_2^2 \leq ar^2$.)

The vectors $N_{v_n}$ and $L_n$ (their common initial point is the terminal point of $v_n$) determine a $2$-plane (with origin this terminal point of $v_n$).  Let $N_n$ be the unit vector in this $2$-plane orthogonal (with respect to $\cdot$) to $L_n$ nearest $N_n$; therefore, as $n \rightarrow \infty$, the angle between $N_n$ and $N_{v_n}$ approaches zero.  For $n$ large enough, the cosine of the angle between $N_n$ and $v_n$ is greater than $\frac 1 {2 c_{2 q}\sqrt{a}}$.   Let $B_2(w_n, R)$ denote the $\|\cdot\|_2$-ball of radius $R$ around $w_n$; likewise, let $B_q(w_n, R)$ be the $\|\cdot\|_q$-ball.  From the proceeding, there is a factor $k_0 > 0$, depending only on $\frac 1 {2 c_{2 q}\sqrt{a}}$ (i.e. $q$), such that $B_2(w_n, k_0 \|w_n - v_n\|_2)$ does not meet $\g v_n$ for all $\g \geq 0$.\footnote{The vectors $v_n$, $N_n$, and $L_n$, in general, determine a $3$-space (with origin the terminal point of $v_n$), but the worst case would be when they determine a $2$-space; see Figure~\ref{figworstcase}, which shows, via planar trigonometry, that $k_0$ is independent of $\|w_n - v_n\|_2$ once the angle is bounded.   Also note that the intersection of $B_2(w_n, R)$ with this $3$-space is a $3$-dimensional $\|\cdot\|_2$-ball with around $w_n$ with radius $R$.  Since $\g v_n$ lies in this $3$-space, $B_2(w_n, R)$ meets $\g v_n$ if and only if the intersection of $B_2(w_n, R)$ with this $3$-space meets $\g v_n$.}  

\begin{figure}[h] 
   \centering
   \includegraphics[width=1.0\textwidth]{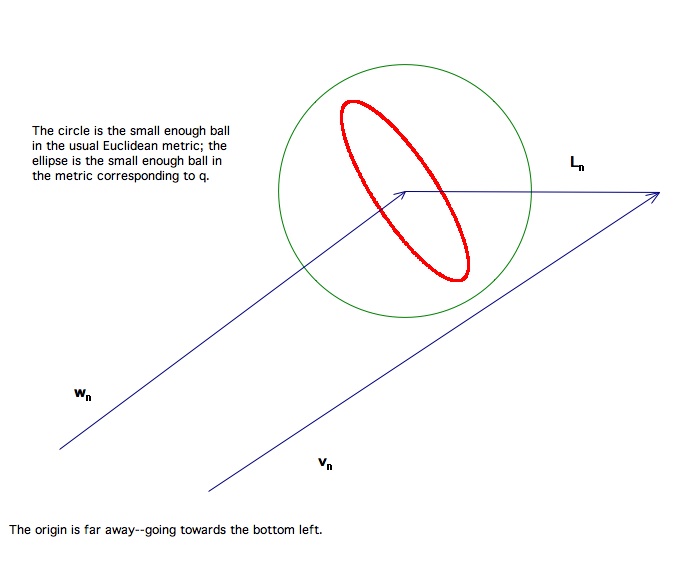} 
    \caption{The worst case.}
   \label{figworstcase}
\end{figure}

Therefore, $B_2(w_n, \frac{k_0}{c_{2q}} \|w_n - v_n\|_q)$ does not meet $\g v_n$ for all $\g \geq 0$.  Further dividing the radius by $c_{2q}$ allows us to inscribe $B_q(w_n, \frac{k_0}{c^2_{2q}} \|w_n - v_n\|_q)$ into $B_2(w_n, \frac{k_0}{c_{2q}} \|w_n - v_n\|_q)$ (see Figure~\ref{figworstcase}).  Finally, it follows that $\|w_n - \g v_n\|_q \geq \frac{k_0}{c^2_{2q}} \|w_n - v_n\|_q$ for all $\g\geq 0$ and all $n$ large enough.  Letting $n \rightarrow \infty$ in (\ref{eqnNormConvert2}) yields $\frac{k_0}{c^2_{2q}} \leq 0$, a contradiction.

\end{proof}

\begin{rema}
If $q$ is a positive-definite form where all the coefficients are the same real number, then the lemma is obvious since $S_r$ is a $k-1$-sphere.  That any two vectors $v \not\sim w \in S_r$ lie on some great circle can be seen by intersecting with the obvious plane.  Then planar trigonometry yields the lemma.
\end{rema}

The analogous lemma holds also for the sup norm and other like norms:  explicitly, for the norm $\|\cdot\|_s$ defined as follows: $\|v\|_s = \max(\b_1 |v_1|, \cdots, \b_k |v_k|)$ for some positive real numbers $\b_i$ and where the $v_i$s are the components of $v$.  
\begin{lemm}\label{lemmCubeMissesLine}  
Let $r>0$ be a real number.  Then there exists a real number $c>0$, depending only on $\|\cdot\|_s$, such that, for all $w \not\sim v \in \{u \in \RR^k \mid \|u\|_s = r\}$ and all real  numbers $\g \geq 0$, we have \[\|w - \g v\|_s > c \|w - v \|_s.\]
\end{lemm}

\begin{proof}
The first part of the proof is identical to that of Lemma~\ref{lemmEllipsoidMissesLine}.  Starting with letting $w_0 = v_0$, we simplify as follows.  Now $w_0$ lies on some face $\mF$ of $S:=\{u \in \RR^k \mid \|u\|_s = r\}$ ($S$ is the $k-1$-dimensional shell of a $k$-dimensional box in $\RR^k$) .  The face is compact; thus we may assume that our convergent subsequence lies only in $\mF$ by taking a subsequence.  The outward unit normal vector $N$ of the affine hyperplane which is determined by $\mF$ is one of the standard basis vectors of $\RR^k$.  The angle computation involving the dot product for $u \in \mF$ reduces to \[\frac{u/r}{\|u/r\|_2} \cdot N  \geq  \frac {1/\b} {\|u/r\|_2} \geq \frac 1 {c}\] where $\b:= \max(\b_1, \cdots, \b_k)$ and the positive constant $c$ exists because $\|\cdot\|_s$ and $\|\cdot\|_2$ are equivalent.

Now $L_n$ is a segment in $\mF$ (since the face is convex), which is orthogonal to $N$.  The desired result now follows in the analogous way.\end{proof}

Let $d\geq2$ be a natural number and $H:= \{w \in \RR^{d} \mid Q(w) = m\}$ for some $m \in \RR \backslash \{0\}$ where $Q:=\a_1Y_1^2 + \cdots \a_d Y_d^2$ is a nondegenerate quadratic form (it does not matter whether this form is positive definite or indefinite).  Since $m$ is never zero, $H$ has no singular points and thus is a hypersurface or, equivalently, a codimension 1, closed, isometrically embedded submanifold (under inclusion) of $\RR^d$ with Riemannian metric induced by the usual dot product on $\RR^d$ (which corresponds to the norm $\|\cdot\|_2$).\footnote{Since $H$ is locally the graph of some smooth function, the results that follow in the rest of this section can be proved for any hypersurface $H$ using very similar proofs, which are, essentially, applications of Taylor approximation.  Note that the local ``shape'' of any hypersurface is approximated by the graph of a quadratic polynomial with coefficients involving the principle curvatures.}  For any point $p \in H$, there are exactly two unit normal vectors, and these lie on the same (affine) line $\mN_p$ in $\RR^d$.  Also, let $B_2(p, r)$ denote the closed ball of $\RR^d$ around the point $p \in \RR^d$ of radius length $r$ given with respect to the norm $\|\cdot\|_2$ and $\partial B_2(p, r)$, its boundary sphere.

For use in this paper, we need only local results: let us fix some large closed ball around the origin of $\RR^d$ and intersect it with $H$; denote this intersection by $H_0$.  We use the fact that $H_0$ is compact to simply the proofs, but, for quadratic varieties, one can make the same statements for $H$.

\begin{lemm}\label{lemm2PlaneIntH}
Let $\mP$ denote a $2$-plane containing $\mN_p$.  There exists a real number $R >0$ (independent of both $\mP$ and $p$) such that for all $p \in H_0$ and all $\mP$, we have $\partial B_2(p, R) \cap \mP  \cap H$ is two distinct points lying on distinct half-planes of $\mP \backslash \mN_p$. 
\end{lemm}

\begin{proof}
Let $H_1$ be the intersection of a larger closed ball around the origin with $H$ than that ball from $H_0$; thus $H_0 \subset H_1$.  Since $H$ is isometrically embedded in $\RR^d$, it does not contain any self-intersections; therefore, for any $p \in H_0$, there exists a positive distance (in $\RR^d$) between it and any intersection of $\mN_p \backslash \{p\}$ with $H$.\footnote{A self-intersection of an immersed manifold occurs when the distance in the ambient manifold of two distinct points of the immersed manifold is zero.}  If there is no universal positive bound for this distance for all $p \in H_0$, then there exists a sequence of pairs $(p,q) \in H_0 \times H_1$ such that this normal distance is approaching zero.  Since $H_0 \times H_1$ is compact and metrizable, there is a convergent subsequence $(p_n,q_n) \in H_0 \times H_1$ for $n \in \NN$; hence there exists a subsequential limit $p_0$ and $q_0 \in H_0$ whose normal distance is zero--a self-intersection point of $H$.\footnote{We assert that $p_0$ and $q_0$ are distinct points of $H$ (its geodesic distance in $H$ cannot be zero).  Assume $p_0=q_0$.  Then both $p_n$ and $q_n$ approach $p_0$ (which implies that $p_n$ and $q_n$ are asymptotically on the same path-component of $H$).  Locally at $p\in H$, $H$ is the graph of a smooth function.  Taylor approximation says  that the difference between $H$ and $T_{p}H$ shrinks quadratically as the Euclidean ball around $p$ shrinks linearly, and the shrinking depends only on the derivatives, so that in any small enough ball inside the shrinking ball, we have quadratic shrinking. Consider a shrinking Euclidean ball around $p_0$; since this ball will contain some pair $(p_n, q_n)$, we have $\mN_{p_n}$ approaches being parallel with an element of $T_{p_n}H$, a contradiction.}     This is a contradiction.

Now, for every point $p \in H_0$, there exists a positive radius $R$ small enough so that $B_2(p,R) \cap H$ is a closed ball of $H$ (the isometric embedding implies that we have the subspace topology on $H$), which is diffeomorphic, under some smooth map $\varphi$, to a closed ball of $d-1$-dimensional Euclidean space; and thus the boundary sphere maps injectively into $H$ under $\varphi^{-1}$.  Now $B_2(p,R)\cap \mP \cap H$ is an arcsegment of a $1$-dimensional smooth manifold (i.e. a smooth curve).  By the proceeding paragraph, if necessary, shrink $R$ below the positive universal bound on normal distance so that this smooth curve does not meet $\mN_p \backslash \{p\}$ in $B_2(p,R)$; therefore, $\varphi(B_2(p,R)\cap \mP \cap H)$ is still a smooth curve in $d-1$-dimensional Euclidean space with distinct endpoints on the boundary of $\varphi(B_2(p,R))$, and $R$ does not depend on $\mP$.  Since the curve does not meet $\mN_p \backslash \{p\}$, its two endpoints must lie on distinct half-planes of $\mP \backslash \mN_p$.

If we replace $R$ with a smaller positive real number in the proceeding paragraph, the same result holds.  Therefore, let $R_p$ be the supremum over all such radii for the point $p \in H_0$.  If there is no universal lower bound on $R_p$ for all $p \in H_0$, then there exists a sequence of $p \in  H_0$ such that $R_p \rightarrow 0$.  Since $H_0$ is compact,  a subsequence converges to a limit $p \in H_0$ such that $R_p=0$, a contradiction of the previous paragraph.  Therefore, there exists a positive $R$ independent of both $p \in H_0$ and $\mP$.
\end{proof}

Now consider an affine hyperplane $\eLL$ in $\RR^d$ (such hyperplanes are always codimension $1$) and its $\varepsilon > 0$ thickening $\eLL^{(\varepsilon)}$.  On a small enough ball, $H_0$ does not leave a slightly thickened hyperplane:

\begin{lemm}\label{lemmClosetoLinear}
For every (small) $\varepsilon >0$, there exists a $R >0$ such that, for all $p \in H_0$ and all $0 < r \leq R$, we have \[B_2(p, r) \cap H \cap \RR^d \backslash\eLL^{(\varepsilon r)} = \emptyset\] where $\eLL$ is the affine hyperplane containing $p$ and whose normal vectors at $p$ all lie in $\mN_p$.
\end{lemm}

\begin{proof}
Let $f := Q - m$ and $p:=(p_1, \cdots, p_d)$.  Then $f \in \sC^\infty(\RR^d)$.  Let $Y:=(Y_1, \cdots, Y_d) \in \RR^d$ (close to the origin).  Then Taylor approximation states that $f(p+Y) = 2(\a_1 p_1Y_1 + \cdots + \a_d p_d Y_d) + (\a_1 Y_1^2 + \cdots + \a_d Y_d^2)$ up to a remainder function $r(Y)$ where \[\lim_{Y \rightarrow \bf{0}} \frac {r(Y)}{\|Y\|_2^2} =0.\]  Also the tangent space at $p$ of $H$ (which is $\eLL$) has equation $0=T(p+Y) = 2(\a_1 p_1Y_1 + \cdots + \a_d p_d Y_d)$.  Since $m \neq 0$, there is some sup-norm ball in $\RR^d$ around the origin of positive radius that does not meet $H$.  Thus the component of $p$ largest in absolute value is bigger than some positive constant depending on $Q$ and $m$ only; without loss of generality, we may assume that $p_d$ is this component.  A point $p + Y$ is in $H$ if and only if $f(p+Y)=0$.  Thus, with ($p$ regarded as the origin), $p+Y$ has coordinates $(Y_1, \cdots, Y_{d-1}, -\frac {1}{2\a_d p_d }[2(\a_1 p_1Y_1 + \cdots + \a_{d-1} p_{d-1} Y_{d-1}) +(\a_1 Y_1^2 + \cdots + \a_d Y_d^2) + r(Y)])$.  The corresponding point on $\eLL$ has coordinates  $(Y_1, \cdots, Y_{d-1}, -\frac {1}{\a_d p_d }(\a_1 p_1Y_1 + \cdots + \a_{d-1} p_{d-1} Y_{d-1})).$  Thus the least distance between $\eLL$ and $H$ is bounded from above by $C\|Y\|^2_2$ for some constant $C \geq 1$ (depending on $Q$ and $m$, for $\|Y\|$ small enough, and independent of $p$).

Let $\mP$ be a $2$-plane containing $\mN_p$ (see Figure~\ref{figZoomupVar}).  By Lemma~\ref{lemm2PlaneIntH}, for small enough $R$ (independent of $\mP$ and $p$), we have $\partial B_2(p, R) \cap \mP  \cap H$ is two distinct points $p_1$ and $p_2$, which lie on distinct half-planes of $\mP \backslash \mN_p$.  Choose $R \leq \frac \varepsilon C$.  Then, by the previous paragraph, the distance between $p_1$ and its corresponding point on $\eLL$ is bounded from above by $\varepsilon R$; likewise, the distance between $p_2$ and its corresponding point on $\eLL$ is bounded from above by $\varepsilon R$.\footnote{The least distance between a vector (with initial point $p$) and a hyperplane through $p$ is given by orthogonal projection onto $\mN_p$.}   The only way for $p_1$ and $p_2$ to be near each other is if they are both close to one of the intersection points of $\mN_p$ with $\partial B_2(p, R) \cap \mP$.  Since $\varepsilon$ is small, these points cannot be close to each other.

The results in the previous paragraph are still true if one replaces $R$ with any positive real number $r \leq R$.  

\begin{figure}[h] 
   \centering
   \includegraphics[width=1.0\textwidth]{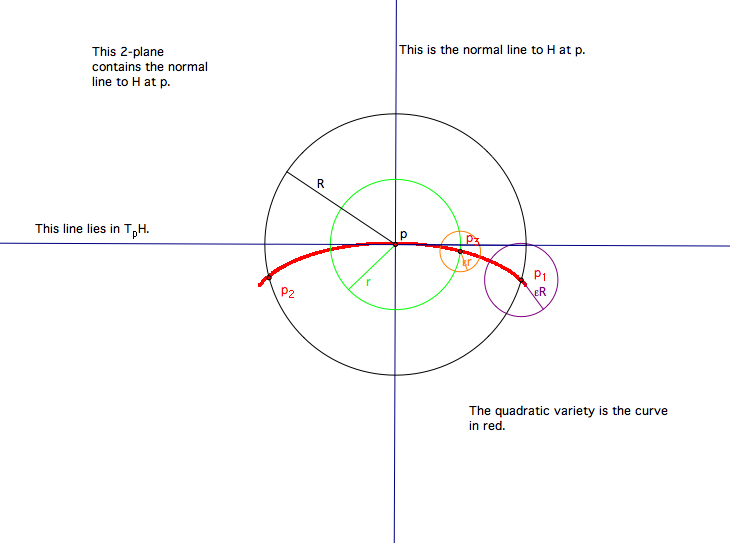} 
    \caption{A zoom-up of the quadratic variety $H$.}
   \label{figZoomupVar}
\end{figure}

Now pick a point $p_3  \in H \cap  B_2(p, R) \backslash \{p\}$.  For some $0 < r \leq R$, it lies on $\partial B_2(p,r)$ and on some $2$-plane $\mP$ containing $\mN_p$.  Consequently, by the above, $p_3$ lies in $\eLL^{(\varepsilon r)}$, which is still true if $r$ is replaced by an even smaller positive real number.  The desired conclusion follows.

\end{proof}

\end{document}